\documentclass[a4paper,10pt,twoside]{article}

\usepackage{amsmath}        
\usepackage{amssymb}
\usepackage{amsthm}
\usepackage[nooneline,hang]{subfigure}
\usepackage[font=normalsize,singlelinecheck=off,labelfont=bf]{caption}
\usepackage{graphicx}
\usepackage{bbold}
\usepackage[usenames]{color}
\usepackage{dsfont}
\usepackage{floatflt}
\usepackage{enumerate}
\usepackage{mathrsfs}
\usepackage{url}
\usepackage{appendix}
\usepackage{fancyhdr}

\newcommand{\R}{\mathds R}
\newcommand{\C}{\mathds C}
\newcommand{\N}{\mathds N}
\newcommand{\Z}{\mathds Z}
\newcommand{\Sp}{\mathds S}
\newcommand{\Od}{{\cal O}}
\newcommand{\eps}{\varepsilon}
\renewcommand{\Re}{\text{Re}}
\renewcommand{\Im}{\text{Im}}

\renewcommand{\sectionmark}[1]%
{\markright{#1}}
\renewcommand{\subsectionmark}[1]{}

\theoremstyle{plain}
\newtheorem{theorem}{Theorem}[section]
\newtheorem{proposition}[theorem]{Proposition}
\newtheorem{corollary}[theorem]{Corollary}
\newtheorem{lemma}[theorem]{Lemma}

\theoremstyle{definition}
\newtheorem{remark}[theorem]{Remark}
\newtheorem{example}[theorem]{Example}

\newtheorem*{theorem*}{Theorem}
\newtheorem*{proposition*}{Proposition}
\newtheorem*{lemma*}{Lemma}
\newtheorem*{corollary*}{Corollary}
\newtheorem*{remark*}{Remark}

\newtheorem*{observation*}{Observation}
\newtheorem*{example*}{Example}
\newtheorem*{examples*}{Examples}
\newtheorem*{assumption*}{Assumption}

\theoremstyle{definition}
\newtheorem{definition}[theorem]{Definition}
\newtheorem*{definition*}{'Definition'}

\newtheorem*{definitionu*}{Definition}

\title{Approximation of conformal mappings by circle patterns}
\author{Ulrike B\"ucking
\thanks{Partially supported by the
DFG Research Unit ``Polyhedral Surfaces''
and by the DFG Research Center {\sc Matheon} ``Mathematics for key
technologies''}
}
\date{}

\begin{document}

\maketitle

\begin{abstract}
A circle pattern is a configuration of circles in the plane whose combinatorics
is given by a planar graph $G$ such that to each vertex of $G$ there
corresponds a circle. If two vertices are connected by an edge in $G$ then the
corresponding circles intersect with an intersection angle in $(0,\pi)$ and
these intersection points can be associated to the dual graph $G^*$.

Two sequences of circle patterns are employed to approximate a given conformal
map $g$ and its first derivative. For the domain of $g$ we use embedded circle
patterns where all circles have the same radius $\eps_n>0$ for a sequence
$\eps_n\to 0$ and where the intersection angles are uniformly bounded. The
image circle patterns have the same combinatorics and intersection angles and
are determined
from boundary conditions (radii or angles) according to the values of $g'$
($|g'|$ or $\arg g'$). The error is of order $1/\sqrt{-\log \eps_n}$.
For quasicrystallic circle patterns
the convergence result is strengthened to
$C^\infty$-convergence on compact subsets and an error of order $\eps_n$.
\end{abstract}

\section{Introduction}
Conformal mappings constitute an important class in the field of complex
analysis.
They may be characterized by the fact that infinitesimal circles are mapped to
infinitesimal circles.
Suitable discrete analogs are of actual interest in the
area of discrete differential geometry and its applications, see~\cite{BSSZ08}.

Bill Thurston first introduced in his
talk~\cite{Thu85} the idea to use finite circles, in particular circle packings,
to define a discrete conformal mapping.
Remember that an {\em embedded planar circle packing} is a configuration of
closed disks with
disjoint interiors in the plane $\C$. Connecting the centers of touching disks
by straight lines yields the {\em tangency graph}.
Let ${\mathscr C}_1$ and ${\mathscr C}_2$ be two circle
packings whose tangency graphs are combinatorially the same. Then
there is a mapping $g_{\mathscr C}:{\mathscr C}_1\to {\mathscr C}_2$ which maps
the centers of circles of ${\mathscr C}_1$
to the corresponding centers of circles of ${\mathscr C}_2$ and is an affine map
on each triangular region corresponding to three mutually tangent circles.
Various connections
between circle packings and classical complex analysis have already been
studied.
A beautiful
introduction and surway is presented by Stephenson in~\cite{St05}.
In particular, several results concerning convergence, i.e.\ quantitative
approximation of conformal mappings by $g_{\mathscr C}$, have been obtained,
see~\cite{RS87,CR92,HeSch96,HeSch98}.

The class of {\em circle patterns} generalizes circle packings as for each
circle packing there is an associated orthogonal circle pattern. Simply add a
circle for each triangular face which passes through the three touching points.
To define a circle pattern we use a planar graph as combinatorial data.
The circles correspond to vertices and the edges specify which circles should
intersect. The intersection angles are given using a labelling
on the edges. Thus an edge corresponds to a kite of two intersecting circles as
in Figure~\ref{figbquad}~(right). Moreover, for interior vertices the kites
corresponding
to the incident edges have disjoint interiors and their union is homeomorphic to
a closed disk.
Similarly as for circle packings, there are results on existence, rigidity, and
construction of special circle pattern. See for
example~\cite{Ri94,Sch97,AB00,BS02,BH03}, where some of the references use a
generalized notion of circle patterns.

Given two circle patterns ${\mathscr C}_1$ and ${\mathscr C}_2$ with the same
combinatorics and intersection angles, define a mapping $g_{\mathscr
C}: {\mathscr C}_1\to {\mathscr C}_2$ similarly as for circle packings. 
Namely, take $g_{\mathscr C}$ to map the centers of circles and the
intersection points of ${\mathscr C}_1$ corresponding to vertices and faces of
$G$ to the corresponding centers of circles and
intersection points of ${\mathscr C}_2$ and extend it to an affine map on each
kite.

For a given conformal map $g$ we use an analytic approach and specify
suitable boundary values for the radius or the angle function according to
$|g'|$ or to $\arg g'$ respectively in order to define the (approximating)
mappings $g_{\mathscr C}$. 
Generalizing ideas of Schramm's convergence proof
in~\cite{Sch97} we obtain convergence in $C^1$ on compact sets if we take for
${\mathscr C}_1^{(n)}$ a sequence of isoradial circle patterns (i.e.\ all radii
are equal) with decreasing radii $\eps_n\to 0$ which approximate
the domain of
$g$. Furthermore, we assume that the intersection angles are uniformly bounded
away from $0$ and $\pi$.
Note in particular that the combinatorics of the circle pattern ${\mathscr
C}_1^{(n)}$ may be irregular or change within the sequence. Thus our
convergence results applies to a considerably broader class of circle patterns
as the known results of Schramm~\cite{Sch97}, Matthes~\cite{Ma05}, or Lan and
Dai~\cite{LD07} for orthogonal circle patterns with square grid combinatorics.

The main idea of the proof is to consider a ``nonlinear discrete Laplace
equation'' for the radius function. This equation turns out to be a (good)
approximation of a known linear Laplace equation and can be used in the case of
isoradial circle patterns to compare discrete and smooth
solutions of the corresponding elliptic problems, that is the logarithm of the
radii of ${\mathscr C}_2^{(n)}$ and $\log|g'|=\Re(\log g')$.
Also, we obtain an a priori estimation of the approximation error
of order $1/\sqrt{\log(1/\eps_n)}$.

If the circle patterns ${\mathscr C}_1^{(n)}$ additionally have only a
uniformly bounded number of different edge directions,
then the corresponding kite patterns are quasicrystallic
rhombic embeddings and the circle patterns ${\mathscr C}_1^{(n)}$ are called
{\em quasicrystallic}~\cite{BMS05}.
For such embeddings we generalize an asymptotic development
given by Kenyon in~\cite{Ke02} of a discrete Green's function. Also, using
similar ideas as Duffin in~\cite{Du53}, we generalize theorems of discrete
potential theory
concerning the regularity of solutions of a discrete Laplace equation. We then
use these results together with the $C^1$-convergence for isoradial circle
patterns and prove
$C^\infty$-convergence on compact sets for a class of quasicrystallic circle
patterns.
In this case the approximation error is
of order $\eps_n$ (or $\eps_n^2$ for square grid or hexagonal
combinatorics and regular intersection angles).
The proof generalizes a method used by He and Schramm in~\cite{HeSch98}.

The article is organized as follows: First we introduce and remind the
terminology and some results
on circle patterns in Section~\ref{CircelPatterns}, focusing in particular on
the radius and the angle function. In Section~\ref{secDirichlet} we formulate
and prove the theorems on $C^1$-convergence for isoradial circle patterns.
After a brief review on quasicrystallic circle patterns in
Section~\ref{secQuasiCirc} we state and prove in
Section~\ref{secConvQuasiCinfty} the theorem on $C^\infty$-convergence.
The necessary results on discrete potential theory are presented
in Appendix~\ref{secPropGreen}.
An extended and more details version of the results can be found in~\cite{diss}.

\thanks{
The author would like to thank Alexander I.~Bobenko, Boris Springborn, and Yuri
Suris for various discussions and helpful advice.}

\section{Circle patterns}\label{CircelPatterns}
To define circle patterns we use combinatorial data and 
intersection angles.

The combinatorics are specified by a {\em b-quad-graph} $\mathscr D$, that
is a strongly regular cell
decomposition of a domain in $\C$ possibly with boundary such that all
2-cells (faces) are embedded  and counterclockwise oriented.
Furthermore all faces of $\mathscr D$ are
quadrilaterals, that is there are exactly four edges incident to
each face, and the 1-skeleton of $\mathscr D$ is a
bipartite graph.
We always assume that the
vertices of $\mathscr D$ are colored white and black. To these two sets of
vertices we associate two planar graphs $G$ and $G^*$ as follows.
The vertices $V(G)$ are all white vertices of 
$V(\mathscr D)$. The edges $E(G)$ correspond to faces of $\mathscr D$, that is
two vertices of $G$ are connected by an edge if and only if they
are incident to the same face of $\mathscr D$. The
dual graph $G^*$ is constructed analogously by taking for
$V(G^*)$ all black vertices of $V(\mathscr D)$.
$\mathscr D$ is called {\em simply
  connected} if it is the cell decomposition of a simply connected
domain of $\C$
and if every closed chain of faces is null homotopic in $\mathscr D$.

For the intersection angles, we use a labelling
$\alpha:F({\mathscr D})\to (0,\pi)$ of the faces of $\mathscr D$. By abuse of
notation, $\alpha$ can also be understood as a function defined on $E(G)$ or on
$E(G^*)$.
The labelling $\alpha$ is called {\em admissible} if it satisfies the following
condition at all interior black vertices $v\in V_{int}(G^*)$:
\begin{equation}\label{condalpha}
  \sum_{f \text{ incident to } v} \alpha(f) = 2\pi.
\end{equation}

\begin{figure}[h]
\begin{center}
 \includegraphics[width=0.45\textwidth]{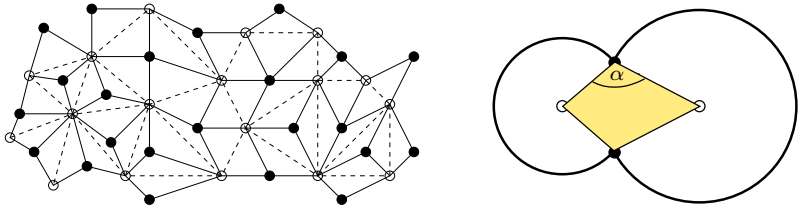}
\hspace{2em}
\input{intersectionAngleKite2.pstex_t}
\end{center}
\caption{{\it Left:} An example of a b-quad-graph $\mathscr D$ (black edges and
bicolored
  vertices) and its associated graph $G$ (dashed edges and white vertices).
{\it Right:} The exterior intersection angle $\alpha$
of two intersecting circles and the associated
  kite built from centers and intersection points.}\label{figbquad}
\end{figure}

\begin{definition}\label{defcircpattern}
Let $\mathscr D$ be a b-quad-graph and let $\alpha:E(G)\to (0,\pi)$
be an admissible labelling.
  An {\em (immersed planar) circle pattern for  $\mathscr D$ (or $G$) and
    $\alpha$} are an indexed collection ${\mathscr 
    C}=\{C_z:z\in V(G)\}$ of circles in $\C$ and an indexed
  collection ${\mathscr K}=\{K_e:e\in E(G)\}=\{K_f:f\in F({\mathscr
    D})\}$ of closed
  kites, which all carry the same orientation, such that the following
conditions hold.
\begin{enumerate}[(1)]
  \item\label{intersectionPoint}
If $z_1,z_2\in V(G)$ are incident vertices in $G$,
the corresponding circles $C_{z_1},C_{z_2}$ intersect with exterior
intersection angle $\alpha([z_1,z_2])$.
 Furthermore, the kite $K_{[z_1,z_2]}$
    is bounded by the centers of the circles $C_{z_1},C_{z_2}$,
    the two intersection points, and the corresponding edges, as in
Figure~\ref{figbquad} (right).
 The intersection points
are associated to black vertices of $V(\mathscr D)$ or to vertices of $V(G^*)$.
\item If two faces are incident in
    $\mathscr D$, then the corresponding kites have one edge
    in common.
  \item Let $f_1,\dots,f_n\in F({\mathscr D})$ be the faces
    incident to an interior vertex $v\in V_{int}({\mathscr D})$. Then 
    the kites $K_{f_1},\dots,K_{f_n}$ have mutually disjoint
    interiors.
Their union $K_{f_1}\cup\dots\cup
    K_{f_n}$ is homeomorphic to a closed disc and contains the point $p(v)$
corresponding to $v$ in its interior.
  \end{enumerate}
  The circle pattern is called {\em embedded} if all kites of
  $\mathscr K$ have mutually disjoint interiors.
 It is called {\em isoradial} if all
  circles of $\mathscr C$  have the same radius.
\end{definition}

  Note that we associate a circle pattern $\mathscr C$ to an immersion of the
kite pattern $\mathscr K$ corresponding to ${\mathscr D}$ where the edges
incident to the same white vertex are of equal length. The kites
can also be reconstructed from the set of circles using the
combinatorics of $G$.
Note further, that there are in general additional intersection points of
circles which are not associated to black vertices of $V({\mathscr D})$.
Some examples are shown in Figure~\ref{figExCirc}.

\begin{figure}[tb]
\subfigure[A regular square grid circle pattern.]{\label{figExQuad}
\includegraphics[height=2.55cm]{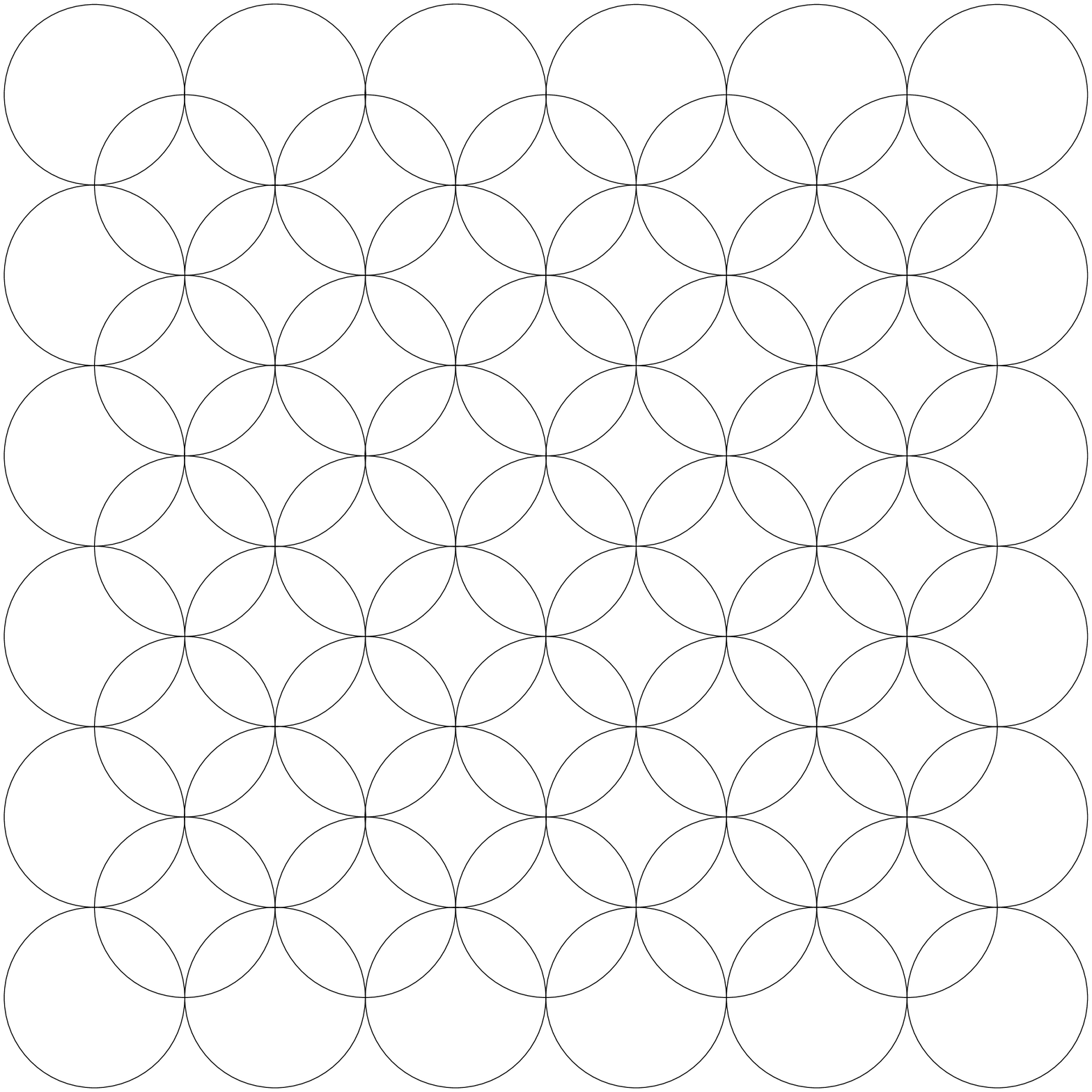}
\hspace{0.05cm}
}
\hspace{0.15cm}
\subfigure[A regular hexagonal circle pattern.]{\label{figExHex}
\includegraphics[height=2.55cm]{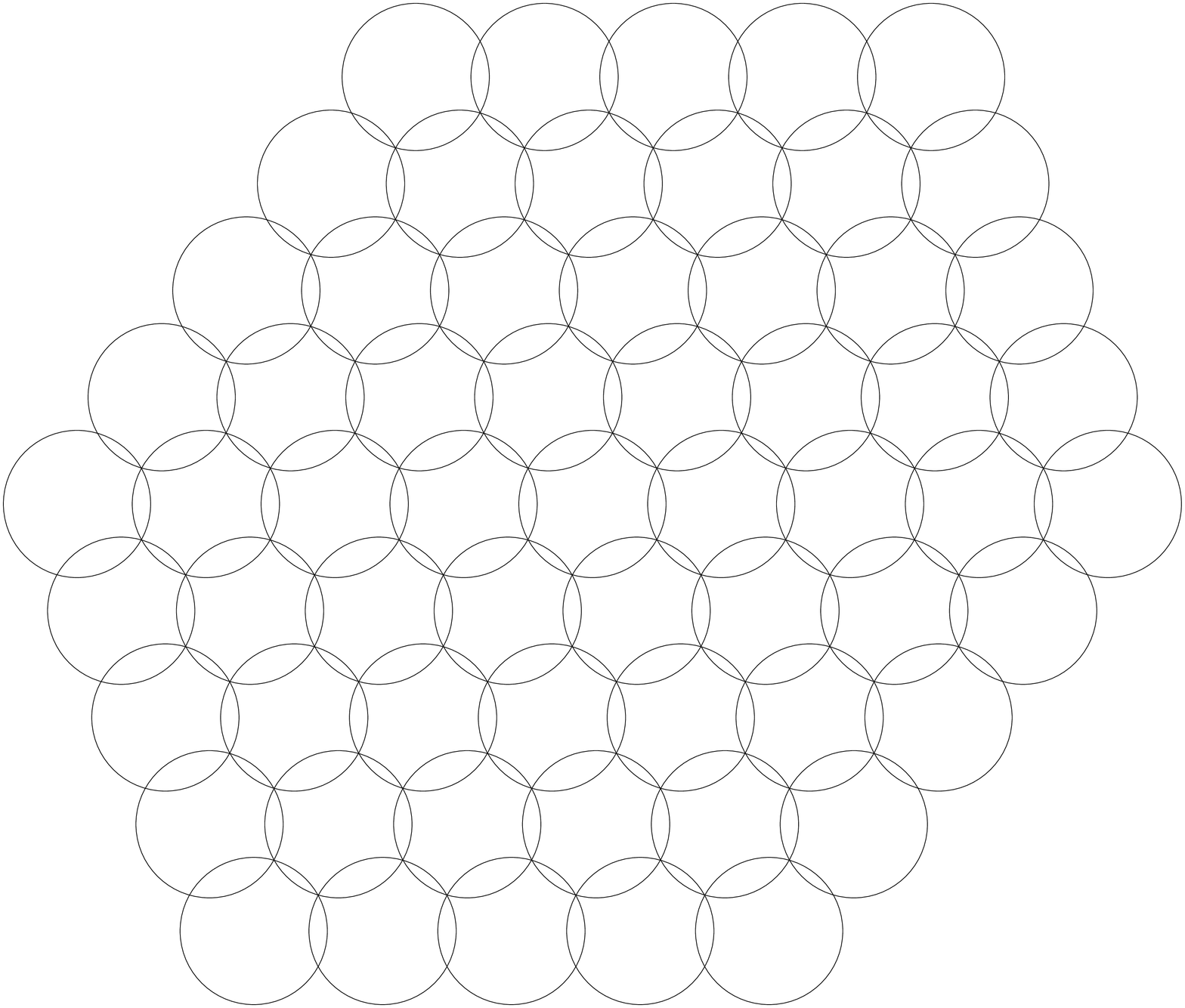}
}
\hspace{0.2cm}
\subfigure[A rhombic embedding (a part of a Penrose tiling) and a
corresponding isoradial circle pattern.]{\label{figPenrose}
\includegraphics[height=2.55cm]{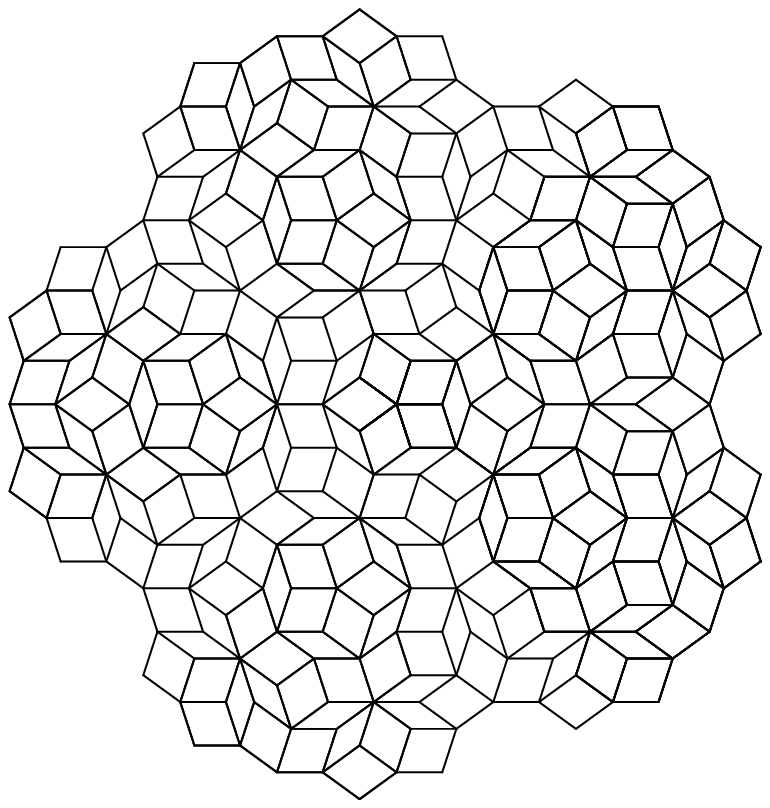}
\includegraphics[height=2.55cm]{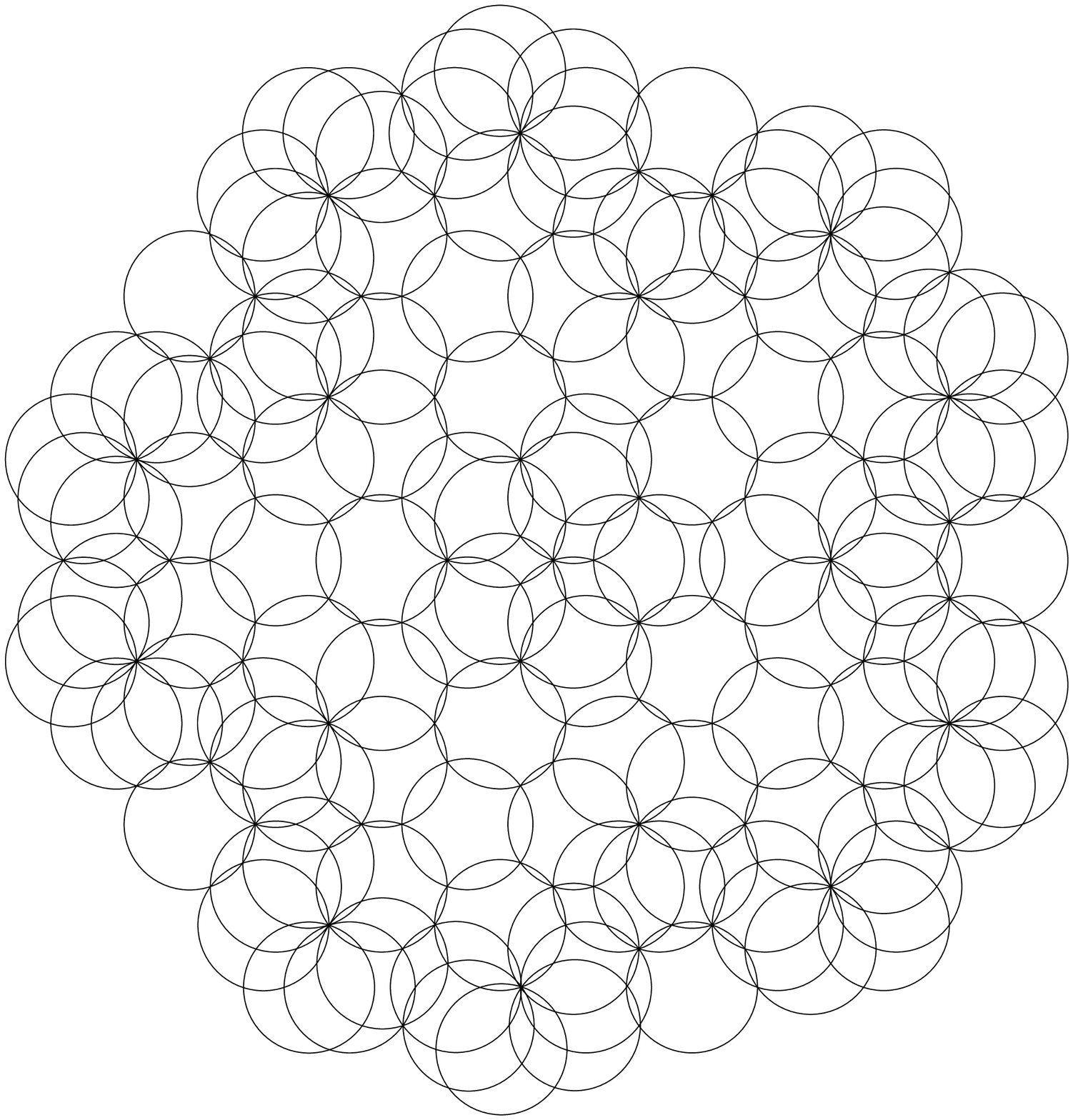}
}
\vspace{-1.5ex}
\caption{Examples of isoradial circle patterns.}\label{figExCirc}
\end{figure}

There are also other definitions for
 circle patterns, for example associated to a {\em Delaunay
 decomposition} of a domain in $\C$. This is a cell decomposition such
that the boundary of each face is a polygon with straight edges which is
inscribed in a circular disk, and these disks have no vertices in their
interior. The corresponding circle pattern can be associated to the
graph $G^*$.
The Poincar{\'e}-dual decomposition of a Delaunay decomposition with the
centers of the circles as vertices and straight edges is a {\em
    Dirichlet decomposition} (or {\em Voronoi diagram}) and corresponds
  to the graph $G$.

Furthermore the definition of circle patterns can be extended allowing
cone-like singularities in the vertices; see~\cite{BS02} and the references
therein.

 \subsection{The radius function}

Our study of a planar circle pattern ${\mathscr C}$ is based on
characterizations and properties of its
radius function $r_{\mathscr C}=r$ which assigns to every vertex $z\in V(G)$
the radius $r_{\mathscr C}(z)=r(z)$ of the corresponding circle $C_z$. The
index $\mathscr C$ will be dropped whenever there is no confusion likely.

The following proposition specifies a necessary and sufficient condition for a
radius function to originate from a planar circle pattern, see~\cite{BS02} for
a proof.
For the special case of orthogonal
circle patterns with the combinatorics of the square grid, there are also
other characterizations, see for example~\cite{Sch97}.

\begin{proposition}\label{propRadius}
Let $G$ be a graph constructed from a b-quad-graph $\mathscr
      D$ and let $\alpha$ be an admissible labelling. 

      Suppose that ${\mathscr C}$ is a planar circle
      pattern for ${\mathscr D}$ and $\alpha$ with radius function
      $r=r_{\mathscr C}$. Then for 
      every interior vertex $z_0\in V_{int}(G)$ we have
      \begin{equation} \label{eqFgen}
        \Biggl(\sum_{[z,z_0]\in E(G)} 
        f_{\alpha([z,z_0])}(\log r(z)-\log r(z_0))\Biggr) -\pi=0,
        \end{equation}
       where
        \[ f_\theta(x):=\frac{1}{2i}\log
        \frac{1-e^{x-i\theta}}{1-e^{x+i\theta}},\]
        and the branch of the logarithm is chosen such that
        $0<f_\theta(x)<\pi$. 

Conversely, suppose that $\mathscr D$ is simply connected and
      that $r:V(G)\to(0,\infty)$ satisfies~\eqref{eqFgen} for
      every $z\in  V_{int}(G)$. Then there is a planar circle pattern
      for $G$ and $\alpha$
whose radius function coincides with $r$.
This pattern is unique up to isometries of $\C$.
\end{proposition}

Note that $2f_{\alpha([z,z_0])}(\log r(z)-\log r(z_0))$ is the angle at $z_0$
of the kite
with edge lengths $r(z)$ and $r(z_0)$ and angle $\alpha([z,z_0])$, as in
Figure~\ref{figbquad} (right).
Equation~\eqref{eqFgen} is the closing
condition for the chain of kites corresponding to the edges incident
to $z_0$ which is condition (3) of
Definition~\ref{defcircpattern}.

For further use we mention some properties of $f_\theta$, see for
example~\cite{Spr03}.

\begin{lemma}\label{lemPropf}
  \begin{enumerate}[(1)]
    \item The derivative of $f_\theta$ is
    $ f_\theta'(x)=\frac{\sin\theta}{2(\cosh x -\cos \theta)}>0$.
    So $f_\theta$ is strictly increasing.
    \item The function $f_\theta$ satisfies the functional equation
      $f_\theta(x)+f_\theta(-x)=\pi-\theta$.
    \item For $0<y<\pi-\theta$ the inverse function of $f_\theta$ is
      $f_\theta^{-1}(y)=\log\frac{\sin y}{\sin(y+\theta)}$.
\end{enumerate}
\end{lemma}

\begin{remark}\label{remInterLap}
Equation~\eqref{eqFgen} can be interpreted as a
nonlinear Laplace equation for the radius function
and is related to a linear discrete Laplacian which is common in
the linear theory of discrete holomorphic functions; see for
example~\cite{Du68,Me,BMS05} for more details. This can be seen as follows.

Let $G$ be a planar graph and let $\alpha$ be an admissible
labelling. Assume there is a smooth one parameter family
of planar circle patterns ${\mathscr C}_\eps$ for $G$ and $\alpha$ with radius
function
$r_\eps$ for $\eps\in(-1,1)$. Then for every interior vertex $z_0$ with incident
vertices $z_1,\dots,z_m$ and all $\eps\in (-1,1)$ Proposition~\ref{propRadius}
implies that
\[\sum_{j=1}^m 
       2 f_{\alpha([z_j,z_0])}(\log r_\eps(z_j)-\log r_\eps(z_0)) = 2\pi .\]
Differentiating this equation with respect to $\eps$ at $\eps=0$, we obtain
\begin{equation}
\sum_{j=1}^m 
       2 f_{\alpha([z_j,z_0])}'(\log r_0(z_j)-\log r_0(z_0))
(v(z_j)-v(z_0)) = 0,
\end{equation}
where $v(z)=\frac{d}{d\eps}\log r_\eps(z)\rvert_{\eps=0}$. Thus 
$v$ satisfies a linear discrete Laplace equation with positive weights.
Lemma~\ref{lemPropf} and a simple calculation show that
\begin{equation}\label{eqLapc}
2 f_{\alpha([z_1,z_2])}'(\log r_0(z_1)-\log r_0(z_2)) =
\left | \frac{p(v_1)-p(v_2)}{p(z_1)-p(z_2)} \right|.
\end{equation}
Here $z_1,z_2\in V(G)$ are two incident vertices which correspond to the
centers of circles $p(z_1),p(z_2)$ of the circle pattern ${\mathscr C}_0$. The
two other corner points of the same kite are denoted by $p(v_1),p(v_2)$.
\end{remark}

In analogy to smooth harmonic functions, the radius function of a planar
circle pattern satisfies a maximum principle and a Dirichlet principle.

\begin{lemma}[Maximum Principle]\label{lemMaxPrinzip}
  Let $G$ be a finite graph associated to a b-quad-graph as above
  with some admissible labelling $\alpha$.
      Suppose ${\mathscr C}$ and ${\mathscr C}^*$ are two planar circle
      patterns for $G$ and $\alpha$ with radius functions
      $r_{\mathscr C},r_{{\mathscr C}^*} :V(G)\to(0,\infty)$.
      Then the maximum and minimum of the quotient $r_{\mathscr
C}/ r_{{\mathscr C}^*}$ is attained at the boundary.
\end{lemma}
A proof can be found in~\cite[Lemma~2.1]{He99}.
If there exists an
isoradial planar circle pattern for $G$ and $\alpha$, the usual maximum
principle for the radius function follows by taking $r_{{\mathscr
 C}^*}\equiv 1$.

\begin{theorem}[Dirichlet Principle]\label{theoDirichlet}
  Let $\mathscr D$ be a finite simply connected b-quad-graph with associated
  graph $G$ and let $\alpha$ be an admissible labelling. 
  
  Let $r:V_\partial(G)\to(0,\infty)$ be some positive function on the
  boundary vertices of $G$. Then $r$ can be extended to $V(G)$ in such
  a way that equation~\eqref{eqFgen} holds at every interior vertex $z\in
  V_{int}(G)$ if and only if
there exists any circle pattern for   $G$ and $\alpha$.
  If it exists, the extension is unique.
\end{theorem}
By lack of a good reference we include a proof.
\begin{proof} The only if part follows directly from the second part of
Proposition~\ref{propRadius}.
  
  To show the if part, assume that there exists a circle pattern for $G$ and
  $\alpha$ with radius function $R:V(G)\to(0,\infty)$. 
   A function $\kappa:V(G)\to(0,\infty)$ which satisfies the inequality
  \begin{equation}
     \left(\sum_{[z,z_0]\in E(G)} 
        f_{\alpha([z,z_0])}(\log \kappa(z)-\log \kappa(z_0))\right) -\pi\geq 0  
  \end{equation}
at every interior vertex $z\in V_{int}(G)$ will be called {\em
  subharmonic} in G. Let $b$ be the minimum of the quotient $r/R$ on
$V_\partial(G)$ and let $\kappa_1$ be equal to $r$ on $V_\partial(G)$
and to $bR$ on $V_{int}(G)$. Then $\kappa_1$ is clearly subharmonic.
The maximum of $\kappa_1/R$ is attained
at the boundary, which is a simple generalization of the Maximum
Principle~\ref{lemMaxPrinzip}.
Let $r^*$ be the supremum of all subharmonic
functions on $G$ that coincide with $r$ on $V_\partial(G)$. Thus $r^*$ is
bounded from above by the maximum 
of $r/R$ on $V_\partial(G)$ which is finite. One easily checks
that $r^*$ satisfies condition~\eqref{eqFgen}.

 The uniqueness claim follows directly from the Maximum
Principle~\ref{lemMaxPrinzip}.
\end{proof}

\subsection{The angle function and relations to the radius function}

Similarly as for polar coordinates of the complex plane, 
a suitably defined angle function can be interpreted as a ``dual'' to the
radius function. We focus on connections between these functions and on
a characterization for angle functions of circle patterns similar to
Proposition~\ref{propRadius}.

Let $G$ be a finite graph associated to a b-quad-graph $\mathscr D$
and let $\alpha$ be an admissible labelling.
Denote by $\vec{E}(\mathscr D)$ the set of
oriented edges, where each edge of $E(\mathscr D)$ is replaced by two oriented
edges of opposite orientation.
Let ${\mathscr C}$ be a planar circle pattern for $\mathscr D$ and $\alpha$.
Define an {\em angle function} $\varphi_{\mathscr C}=\varphi$
on $\vec{E}(\mathscr D)$ as follows.
Denote by $p(w)$ the point of $\mathscr C$ corresponding to the vertex $w\in
V(\mathscr D)$. For $\vec{e}=\overrightarrow{zv}\in \vec{E}(\mathscr D)$
set $\varphi(\vec{e})=\arg (p(v)-p(z))$ to be the argument of $p(v)-p(z)$, that
is the angle between the positively oriented real axis and the vector
$p(v)-p(z)$. As this argument is only unique up to addition of multiples of
$2\pi$, we will mostly consider $\varphi\in\R/(2\pi\Z)$.
But note that $e^{i\varphi(\vec{e})}$ is well defined. Furthermore
\begin{equation}\label{eqdiffedges}
\varphi(\vec{e})-\varphi(-\vec{e})= \pi \pmod{2\pi}.
\end{equation}

Our choice of the angle function leads to the following 
connections between radius function $r$ and angle function $\varphi$ for a
planar circle pattern for $\mathscr D$ and $\alpha$.

Let $f\in F({\mathscr D})$ be a face of ${\mathscr D}$.
Without loss of generality, we assume that the notation for the vertices and
edges of $f$ is taken from
Figure~\ref{anglesFig} (left). More precisely, the white vertices $z_-$ and
$z_+$ of $f$ and the black vertices $v_-$ and $v_+$ are labelled
such that the points $z_-, v_-,z_+,v_+$ appear in this cyclical order using the
 counterclockwise orientation of $f$. Furthermore the edges are labelled
such that $\vec{e}_1=\overrightarrow{z_-v_+}$,
$\vec{e}_2=\overrightarrow{z_-v_-}$, $-\vec{e}_3=\overrightarrow{v_-z_+}$,
$-\vec{e}_4=\overrightarrow{v_+z_+}$.
Then  the following equations hold.
  \begin{alignat}{3}
    &\varphi(\vec{e}_1)-\varphi(-\vec{e}_3) &&= \alpha(f)-\pi +
    2f_{\alpha(f)}(\log r(z_+)-\log r(z_-))& \pmod{2\pi} \label{eqCR1} \\
    &\varphi(-\vec{e}_4)-\varphi(\vec{e}_2) &&=  \alpha(f)-\pi +
    2f_{\alpha(f)}(\log r(z_+)-\log r(z_-))& \pmod{2\pi}  \\
    &\varphi(\vec{e}_1)-\varphi(\vec{e}_2) &&=
    2f_{\alpha(f)}(\log r(z_+)-\log r(z_-)) &\pmod{2\pi} \label{eqCR3} \\
    &\varphi(-\vec{e}_4)-\varphi(-\vec{e}_3) &&=
    -2f_{\alpha(f)}(\log r(z_-)-\log r(z_+)) &\pmod{2\pi} \\
    &\varphi(-\vec{e}_3)-\varphi(\vec{e}_2) &&= \pi - \alpha(f)
    &\pmod{2\pi} \label{eqCR5} \\
    &\varphi(-\vec{e}_4)-\varphi(\vec{e}_1) &&=\alpha(f) -\pi
    &\pmod{2\pi} \label{eqCR6}
  \end{alignat}
Additionally, the angle function $\varphi$ satisfies the following
\begin{lemma}[Monotonicity condition]\label{MonoCond}
  Let $z\in V(G)$ be a white vertex and let
  $e_1,\dots,e_n$ be the sequence of all incident edges in
  $E(\mathscr D)$ which are cyclically ordered respecting the
counterclockwise orientation of the circle $C_z$. Then the values of
$\varphi\in\R$ can be changed by suitably adding
  multiples of $2\pi$ such that $\varphi(\vec{e}_j)$ is an increasing
  function of the index $j$ and if $z$ is an interior vertex, then
  also $\varphi(\vec{e}_j)-\varphi(\vec{e}_1)<2\pi$
  for all $j=1,\dots, n$. 
\end{lemma}
\begin{figure}[tb]
\begin{center}
\setlength{\unitlength}{0.9cm}
\begin{picture}(6,3)(-0.5,0)
\put(0.5,1.5){\circle{0.5}}
\put(4.5,1.5){\circle{0.5}}
\put(2.5,2.5){\circle*{0.5}}
\put(2.5,0.5){\circle*{0.5}}
\thicklines
\put(0.7,1.65){\vector(2,1){1}}
\put(0.7,1.65){\line(2,1){1.54}}
\thicklines
\put(2.75,2.4){\vector(2,-1){1}}
\put(2.75,2.4){\line(2,-1){1.54}}
\thicklines
\put(0.7,1.35){\vector(2,-1){1}}
\put(0.7,1.35){\line(2,-1){1.54}}
\thicklines
\put(2.75,0.57){\vector(2,1){1}}
\put(2.75,0.57){\line(2,1){1.55}}
\put(2.4,1.95){$\alpha$}
\put(2.4,0.9){$\alpha$}
\put(1,1.4){$\beta_-$}
\put(3.6,1.4){$\beta_+$}
\put(-0.4,1.4){$z_-$}
\put(4.95,1.4){$z_+$}
\put(2.4,0){$v_-$}
\put(2.4,2.9){$v_+$}
\put(1.3,2.3){$\vec{e}_1$}
\put(1.3,0.5){$\vec{e}_2$}
\put(3.4,2.3){$-\vec{e}_4$}
\put(3.4,0.5){$-\vec{e}_3$}
\end{picture}
\hspace{1cm}
\input{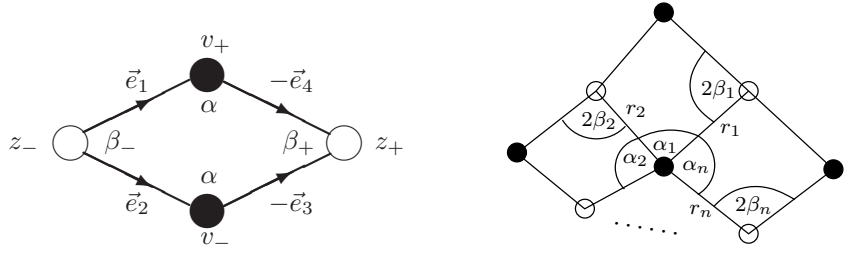}
\end{center}
\caption{
{\em Left:} A face of $\mathscr D$ with oriented edges.
 {\em Right:} An interior intersection point with its neighboring faces.
}\label{anglesFig} 
\end{figure}

The following theorem is useful to compare two circle patterns with the same
combinatorics and intersection angles.

\begin{theorem}\label{theoCompare}
  Let $\mathscr D$ be a b-quad-graph and let $\alpha$ be an admissible
  labelling. 

Let  ${\mathscr C}$ and $\hat{\mathscr C}$ be two planar circle 
  patterns for $\mathscr D$ and $\alpha$ with radius functions
  $r_{{\mathscr C}}=r$ and $r_{\hat{\mathscr C}}=\hat{r}$ and
  angle functions $\varphi_{\mathscr C} =\varphi$ and
  $\varphi_{\hat{\mathscr C}} =\hat{\varphi}$ respectively.
Then the difference $\hat{\varphi}-\varphi$
gives rise to a function $\delta:V(G^*)\to \R$ such that
the following
condition holds on every face $f\in F(\mathscr D)$. 
\begin{equation}\label{eqdeltaw}
2f_{\alpha(f)}\left(\log\left({\textstyle
      \frac{w(z_+)}{w(z_-)}}\right)+\log\left({\textstyle
  \frac{r(z_+)}{r(z_-)}}\right)\right)  -
2f_{\alpha(f)}\left(\log\left({\textstyle \frac{r(z_+)}{r(z_-)}}\right)\right) =
\delta(v_+)-\delta(v_-) 
\end{equation}
Here we have defined $w:V(G)\to\R$, $w(z)=\hat{r}(z)/r(z)$ and the
notation is taken from Figure~\ref{anglesFig}~(left) as above.

  Conversely, assume that ${\mathscr C}$ is a planar circle 
  pattern for $\mathscr D$ and $\alpha$ with radius functions
  $r_{{\mathscr C}}=r$  and
  angle function $\varphi_{\mathscr C} =\varphi$. Let $\delta:V(G^*)\to \R$ and
  $w:V(G)\to\R_+$ be two functions which
  satisfy equation~\eqref{eqdeltaw} for every face $f\in F(\mathscr
  D)$. Then $(rw)$ and $(\varphi+\delta)$ are 
  the radius and angle function of a planar circle 
  pattern for $\mathscr D$ and $\alpha$.
 This circle pattern is unique up to translation.
\end{theorem}
\begin{proof}
If ${\mathscr C}$ and $\hat{\mathscr C}$ are two planar circle
  patterns for $\mathscr D$ and $\alpha$,
equations~\eqref{eqdiffedges}, \eqref{eqCR5}, and \eqref{eqCR6} imply
that the difference $\hat{\varphi}-\varphi$ is constant for all edges
incident to any fixed black vertex $v\in V(G^*)$. Therefore
$\delta\pmod{2\pi}$ is well defined on the intersection points and
encodes the relative rotation of the star of edges at $v$. Also
equation~\eqref{eqdeltaw} holds modulo $2\pi$.

To obtain a
function $\delta$ with values in $\R$, fix $\delta(v_0)\in[0,2\pi)$
for one arbitrary vertex $v_0\in V(G^*)$ (for each connected component of
$G^*$). Define the values of $\delta$ for all
incident vertices in $G^*$ by equation~\eqref{eqdeltaw}.
Continue this construction
until a value has been assigned to all vertices.
This proceedure leads to a well-defined function,
as by Proposition~\ref{propRadius} the sum of the left hand side of
equation~\eqref{eqdeltaw} is zero for
simple closed paths in $E(G^*)$ around a white vertex.

To prove the converse claim, observe that if
 $r$ is a radius function, equation~\eqref{eqdeltaw} implies
  that $(wr)$ fullfills equation~\eqref{eqFgen}. 
By Proposition \ref{propRadius} there is a circle pattern with radius
function $(wr)$.
Adjust the rotational freedom at one edge according to $(\varphi+\delta)$.
Equation~\eqref{eqdeltaw} implies that $(\varphi+\delta)$ is indeed the angle
function of this circle pattern.
\end{proof}

The preceeding theorem motivates the definition of a comparison
function for two circle patterns with the same combinatorics and intersection
angles.

Let $G$ be a graph associated to a b-quad-graph $\mathscr
      D$ and let $\alpha$ be an admissible labelling. 
      Suppose that ${\mathscr C}_1$ and ${\mathscr C}_2$ are planar circle
      patterns for ${\mathscr D}$ and $\alpha$ with radius functions
      $r_{{\mathscr C}_1}$ and $r_{{\mathscr C}_2}$ and angle functions
$\varphi_{{\mathscr C}_1}$ and $\varphi_{{\mathscr C}_2}$ respectively. Let
$\delta:V(G^*)\to \R$ be a function corresponding to
$\varphi_{{\mathscr C}_2} -\varphi_{{\mathscr C}_1}$ as in
Lemma~\ref{theoCompare}.
Define a {\em comparison function} $w:V({\mathscr D})\to \C$ by
\begin{equation}\label{eqdefw}
\begin{cases} w(y)=r_{{\mathscr C}_2}(y)/r_{{\mathscr C}_1}(y) &\text{for }
y\in V(G),\\
w(x)=\text{e}^{i\delta(x)}\in \Sp^1 & \text{for } x\in
V(G^*). \end{cases}
\end{equation}
Note that $w(y)$ is the scaling factor of the circle corresponding to $y\in
V(G)$ when changing from
the circle pattern ${\mathscr C}_1$ to ${\mathscr C}_2$.
$w(x)$ gives the rotation of the edge-star at $x\in V(G^*)$. 
Furthermore, $w$ satisfies the following {\em Hirota Equation} for
all faces $f\in F({\mathscr D})$.
\begin{equation}\label{eqw}
  w(x_0)w(y_0)a_0 -w(x_1)w(y_0)a_1
  -w(x_1)w(y_1)a_0 +w(x_0)w(y_1)a_1 =0
\end{equation}
Here $x_0,x_1\in V(G^*)$ and $y_0,y_1\in V(G)$ are the black and white
vertices incident to $f$ and $a_0=x_0-y_0$ and $a_1=x_1-y_0$ are the
directed edges. Equation~\eqref{eqw} is the
closing condition for the kite of ${\mathscr C}_2$ which corresponds to the
face $f$.

Angle functions associated to planar circle patterns can be characterized in a
similar way as radius functions are qualified
in Proposition~\ref{propRadius}.

\begin{proposition}\label{propangle}
Let $\mathscr D$ be a b-quad-graph with associated graphs 
  $G$ and $G^*$ and
  let $\alpha$ be an admissible labelling.

      Suppose ${\mathscr C}$ is a planar circle
      pattern for $\mathscr D$ and $\alpha$ with angle function
      $\varphi=\varphi_{\mathscr C}$. Then $\varphi$ satisfies
equations~\eqref{eqdiffedges}, \eqref{eqCR5}, \eqref{eqCR6}, the
      Monotonicity condition~\ref{MonoCond} at every white
      vertex of $\mathscr D$, and the following two conditions.
      \begin{enumerate}[(i)]
        \item
      Let $f$ be a face of $\mathscr D$ and let $e_1$ and $e_2$ be two edges
      incident to $f$ and to the same white vertex and assume that $e_1$ and
$e_2$ are enumerated in clockwise order as in Figure~\ref{anglesFig} (left).
Define an angle $\beta\in(0,\pi)$ by
       \[ 2\beta=\varphi(\vec{e_1})-\varphi(\vec{e_{2}})
      \pmod{2\pi},\]
where the orientation of the edges is chosen such that the vectors point from
a white vertex to a black vertex, as in Figure~\ref{faceFig} (left).
     Then
      \begin{equation}\label{eqanglefacecond}
        \beta+\alpha(f)<\pi.
      \end{equation}
      \item For an interior black vertex $v\in V_{int}(G^*)$, denote by
      $e_1,\dots, e_{n},e_{n+1}=e_1$ all incident edges of $\mathscr D$ in
      counterclockwise order and by $f_j$ the face of $\mathscr D$
      incident to $e_{j}$ and 
      $e_{j+1}$. Denote by $e_j^*$ ($j=1,\dots,n$) the edge incident to
      $e_j$ and $f_j$ which is not incident to $v$.   For
      $j=1,\dots,n$ define as above $\beta_j\in(0,\pi)$ by  
      $ 2\beta_j=\varphi(\vec{e_j})-\varphi(\vec{e_{j}}^*)
      \pmod{2\pi}$, where we choose the same orientation of the edges from
white to black vertices as above.
      Then
      \begin{gather}
        \sum_{j=1}^n f_{\alpha(f_j)}^{-1}(\beta_j) =
        0. \label{eqintersectionPoint} 
      \end{gather}
      \end{enumerate}

Conversely, suppose that $\mathscr D$ is simply connected and
      that $\varphi:\vec{E}({\mathscr D})\to\R/(2\pi\Z)$ satisfies
      equations~\eqref{eqdiffedges}, \eqref{eqCR5}, \eqref{eqCR6}, the 
      Monotonicity condition~\ref{MonoCond},
condition~\eqref{eqanglefacecond} at every white
      vertex of $\mathscr D$, and condition~\eqref{eqintersectionPoint} at every
interior black vertex.
      Then there is a planar circle pattern
      for $\mathscr D$ and $\alpha$ with angle function $\varphi$. 
This pattern is unique up
      to scaling and translation.
\end{proposition}
\begin{proof}
For a given planar circle pattern equations~\eqref{eqCR1}--\eqref{eqCR6} and
Lemma~\ref{MonoCond} hold. 
To show~\eqref{eqanglefacecond},
      consider a kite
      corresponding to a face of $\mathscr D$ as in
Figure~\ref{anglesFig}~(left). Note that $\beta_-=2\beta$ by
equation~\eqref{eqCR3},
$\beta_- +\beta_+ +2\alpha =2\pi$, and
$\beta_-,\beta_+,\alpha>0$.
      Using notation of Figure~\ref{anglesFig} (right)
we also deduce that
       $ f_{\alpha_j}^{-1}(\beta_j)= \log r_{j+1} -\log r_j$ 
      for $j=1,\dots,n$, where we identify $r_1=r_{n+1}$. Now
    \eqref{eqintersectionPoint} follows immediately.

In order to prove the converse claim, we
    construct a radius function $r:V(G)\to(0,\infty)$ and build a circle
pattern corresponding to $r$ and $\varphi$.

    Let $z\in V(G)$ be an interior white vertex. Set
    $r(z)=1$. Consider a neighboring white vertex $z'\in V(G)$ and the
    face $f\in F(\mathscr D)$ incident to $z$ and $z'$. Denote the black
vertices of $\mathscr D$ incident to $z$ and $f$ by $v_1$, $v_2$ such that
$v_1$, $z$, $v_2$, $z'$ appear in counterclockwise order along the boundary of
$f$. Define $\psi_-\in (0,2\pi)$ by
    $\psi_-=\varphi(\overrightarrow{zv_1})-\varphi(\overrightarrow{zv_2})
\pmod{2\pi}$.
Condition~\eqref{eqanglefacecond} implies that
    $r(z'):=r(z)\text{exp}(f_{\alpha(f)}^{-1}(\psi_-/2))$ is well defined
    and positive.
    We proceed in this way until a radius has been
    assigned to all white vertices. Condition~\eqref{eqintersectionPoint}
guarantees that these assignments do not
    lead to 
    different values when turning around a black vertex (see
Figure~\ref{anglesFig} (right) with $\psi=2\beta$). Thus $r$ is
    uniquely determined up to the choice of the initial radius, which
    corresponds to a global scaling.
For each face $f\in F(\mathscr D)$ construct a kite 
with lengths $r(z_+),r(z_-)$ of the edges incident to the white
  vertices $z_+,z_-$ of $f$ respectively and angle $\alpha=\alpha(f)$.
  Lay out one kite fixing the rotational freedom according
  to $\varphi$. Successively add all other kites, respecting
  the combinatorics of $\mathscr D$.
By construction and assumptions,
at every interior vertex the
  angles of the kites having this vertex in common add up to $2\pi$.
Thus we obtain a circle pattern with angle function $\varphi$.
\end{proof}

\section[$C^1$-convergence for isoradial circle patterns]{$C^1$-convergence with
Dirichlet or Neumann \\ boundary conditions}
\label{secDirichlet}

In this section we state and prove our main results on convergence for
isoradial circle patterns. We begin with Dirichlet boundary conditions, that is
we first focus on the radius function with given boundary values.

\begin{theorem}\label{theoConvC1}
Let $D\subset\C$ be a simply connected bounded domain, and let
$W\subset\C$ be open such that $W$ contains the closure $\overline{D}$ of $D$.
Let $g:W\to\C$ be a
locally injective holomorphic function. Assume, for convenience, that
$0\in D$. 

For $n\in\N$ let ${\mathscr D}_n$ be a b-quad-graph with associated
graphs $G_n$ and $G_n^*$ and let $\alpha_n$ be an
admissible labelling. We assume 
that ${\mathscr D}_n$ is simply connected and
that $\alpha_n$ is uniformly bounded such that for all $n\in\N$ and
all faces $f\in F({\mathscr D}_n)$ 
\begin{equation}\label{boundalpha}
|\alpha_n(f)-\pi/2| <C 
\end{equation}
with some constant $0<C<\pi/2$ independent of $n$. 

Let $\varepsilon_n\in(0,\infty)$ be a sequence of positive
numbers such that $\varepsilon_n\to 0$ for $n\to\infty$. For each
$n\in\N$, assume that
there is an isoradial circle pattern for $G_n$ and $\alpha_n$,
    where all circles have the same radius $\eps_n$. Assume further
    that all centers of circles lie in the domain $D$ and that any point
    $x\in\overline{D}$ which is not contained in any of the disks
    bounded by the circles of the pattern has a distance less than
    $\hat{C}\eps_n$ to the nearest center of a circle and to the boundary
    $\partial D$, where $\hat{C}>0$ is some constant independent of
    $n$.
Denote by $R_n\equiv \eps_n$ and $\phi_n$ the radius and the angle function
of the 
above circle pattern for $G_n$ and $\alpha_n$. By abuse of notation,
we do not distinguish 
between the realization of the circle pattern, that is 
the centers of circles $z_n$, the intersection points $v_n$, and the
edges connecting corresponding points in ${\mathscr D}_n$ or $G_n$, and
the abstract b-quad-graph ${\mathscr D}_n$ and the graphs $G_n$ and
$G_n^*$. Also, the index $n$ will be dropped from the notation of the
vertices and the edges.

Define another radius function on $G_n$ as follows.
At boundary vertices $z\in V_{\partial}(G_n)$ set
\begin{equation}
  r_n(z)=R_n(z)\left|g'(z)\right|.
\end{equation}
Using Theorem~\ref{theoDirichlet} extend $r_n$ to a solution
of the Dirichlet problem on $G_n$. 
Let $z_0\in V(G_n)$ be such that the disk bounded by the circle
$C_{z_0}$ contains $0$ and let $e=[z_0,v_0]\in E({\mathscr D}_n)$ be one of
the edges 
incident to $z_0$ such that $\phi_n(\vec{e})\in[0,2\pi)$ is minimal. 

Let $\varphi_n$ be the angle function corresponding to $r_n$ that
satisfies
\begin{equation}\label{eqNormPhi}
 \varphi_n(\vec{e})=\arg\left(g'(v_0)\right)+\phi_n(\vec{e}). 
\end{equation}
Let ${\mathscr C}_n$ be the planar circle pattern
with radius function $r_n$ and angle function $\varphi_n$.
Suppose that ${\mathscr C}_n$ is normalized  by a translation such that
\begin{equation}
  p_n(v_0)=g(v_0),
\end{equation}
where $p_n(v)$ denotes the intersection point corresponding to $v\in
V(G_n^*)$.
For $z\in D$ set 
\[ g_n(z)=p_n(w)\quad \text{and}\quad q_n(z)=\frac{r_n(v)}{R_n(v)}
e^{i(\varphi_n(\overrightarrow{vw}) -\phi_n(\overrightarrow{vw}))},\]
where $w$ is a vertex of  
$V(G_n^*)$ closest to $z$ and $v$ is a vertex of $V(G_n)$ closest to
$z$ such that $[v,w]\in E({\mathscr D}_n)$.
 
Then $q_n\to g'$ and $g_n\to g$ uniformly on compact
subsets in $D$ as $n\to\infty$.
\end{theorem}

\begin{remark}
The proof of Theorem~\ref{theoConvC1} actually shows
the following a priori estimations for the approximating functions $q_n$ and
$g_n$.
\begin{equation*}
\|q_n-g'\|_{V(G_n)\cap K} \leq C_1 (-\log_2\eps_n)^{-\frac{1}{2}}
\ \text{ and }\ \|g_n-g\|_{V(G_n)\cap K} \leq C_2
(-\log_2\eps_n)^{-\frac{1}{2}} 
\end{equation*}
for all compact sets $K\subset D$, where the constants $C_1,C_2$
depend on $K$, $g$, $D$, and on the constants of Theorem~\ref{theoConvC1}.
\end{remark}

We begin with an a priori estimation for the quotients of the radius
functions. 

\begin{lemma}\label{lem1}
  For $z\in V(G_n)$ set 
  \begin{align*}
    & h_n(z)=\log\left|g'(z)\right|, \\
    & t_n(z)=\log( r_n(z)/R_n(z))=\log( r_n(z)/\eps_n).
  \end{align*}
  Then 
  \[ h_n(z)-t_n(z)=\Od(\eps_n).\]
\end{lemma}

Here and below the notation $s_1=\Od(s_2)$ means that there is a
constant $C$ which may depend on $W,D,g$, but not on $n$ and $z$, such
that $|s_1|\leq Cs_2$ wherever $s_1$ is defined. A direct consequence of
Lemma~\ref{lem1} is
\begin{equation}
  r_n(z)=R_n(z)\left|g'(z)\right| + \Od(\eps_n^{2}).
\end{equation}

Our proof uses ideas of Schramm's proof of the corresponding Lemma 
in~\cite{Sch97}.
\begin{proof}
Consider the function 
\begin{equation*}
  p(z)=t_n(z)-h_n(z)+\beta |z|^2,
\end{equation*}
where $\beta\in(0,1)$ is some function of $\eps_n$. We want to choose
$\beta$ such that $p$ will have no maximum in $V_{int}(G_n)$.

Suppose that $p$ has a maximum at $z\in V_{int}(G_n)$. Denote by
$z_1,\dots, z_m$ the incident vertices of $z$ in $G_n$ in counterclockwise
order. Then we have
\begin{equation}\label{eqineq}
  t_n(z_j)-t_n(z)\leq x_j
\end{equation}
for $j=1,\dots,m$ where
\begin{equation}
x_j=h_n(z_j)-h_n(z)-\beta|z_j|^2+\beta|z|^2.
\end{equation}
Since $z\in V_{int}(G_n)$, we have $|z|=\Od(1)$ and by assumption
$z-z_j=\Od(\eps_n)$. With $\beta\in(0,1)$ this leads to
$\beta|z_j|^2-\beta|z|^2 =\Od(\eps_n)$.
Using this estimate and the smoothness of $\Re(\log g')$, we get $x_j
=\Od(\eps_n)$.

From~\eqref{eqineq}, the definition of $t_n(z)=\log (r_n(z)/
R_n(z))=\log r_n(z)-\log \eps_n$ and the monotonicity 
of the sum in equation~\eqref{eqFgen} (see Lemma~\ref{lemPropf}~(i)), we get
\begin{align}
0=\left(\sum_{j=1}^m
        f_{\alpha(z,z_j)}(\underbrace{\log r_n(z_j)-\log
r_n(z)}_{= t_n(z_j)-t_n(z)})\right) -\pi
      \leq \left(\sum_{j=1}^m
        f_{\alpha(z,z_j)}\left(x_j\right)\right)
      -\pi. \label{eqMonoq}
\end{align}
Remembering $x_j=\Od(\eps_n)$, we can consider
a Taylor expansion of the right hand side of inequality~\eqref{eqMonoq} about
$0$ in order to make an $\Od(\eps_n^3)$-analysis.

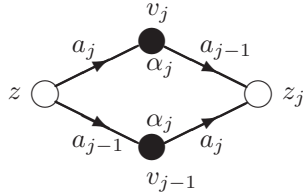
\begin{floatingfigure}[p]{4.3cm}
\begin{center}
\setlength{\unitlength}{0.7cm}
\begin{picture}(6,3)(-0.5,0)
\put(0.5,1.5){\circle{0.5}}
\put(4.5,1.5){\circle{0.5}}
\put(2.5,2.5){\circle*{0.5}}
\put(2.5,0.5){\circle*{0.5}}
\thicklines
\put(0.7,1.65){\vector(2,1){1}}
\put(0.7,1.65){\line(2,1){1.55}}
\thicklines
\put(2.75,2.4){\vector(2,-1){1}}
\put(2.75,2.4){\line(2,-1){1.53}}
\thicklines
\put(0.7,1.35){\vector(2,-1){1}}
\put(0.7,1.35){\line(2,-1){1.55}}
\thicklines
\put(2.75,0.57){\vector(2,1){1}}
\put(2.75,0.57){\line(2,1){1.54}}
\put(2.4,1.95){$\alpha_j$}
\put(2.4,0.9){$\alpha_j$}
\put(-0.2,1.4){$z$}
\put(4.95,1.4){$z_j$}
\put(2.4,-0.2){$v_{j-1}$}
\put(2.4,3.0){$v_j$}
\put(1.0,2.3){$a_{j}$}
\put(1.0,0.5){$a_{j-1}$}
\put(3.4,2.3){$a_{j-1}$}
\put(3.4,0.5){$a_{j}$}
\end{picture}
\end{center}
\caption{
A rhombic face of ${\mathscr D}_n$ with oriented
edges.}\label{faceFig}
\end{floatingfigure}
Consider the chain of faces $f_j$ of ${\mathscr D}_n$ ($j=1,\dots, m$)
which are incident to $z$ and $z_j$.
The enumeration
of the vertices $z_j$ (and hence of the faces $f_j$) and of the black
vertices $v_1,\dots,v_m$ incident to these faces can be chosen such that
$f_j$ is incident to $v_{j-1}$ and $v_j$ for $j=1,\dots, m$, where
$v_0=v_m$. Furthermore, using this enumeration $i(z_j-z)$ and $(v_j-v_{j-1})$
are parallel,
see Figure~\ref{faceFig}.
As each face $f_j$ of an isoradial circle pattern is a
rhombus we can write, using the notation of Figure~\ref{faceFig},
\begin{align}
z_j-z= a_{j-1}+a_j \qquad \text{and} \qquad
v_j-v_{j-1}=a_j-a_{j-1}. \label{eqzva}
\end{align} 
Denoting $\alpha_j =\alpha([z,z_j])$, $l_j=|z_j-z|=2\eps_n\sin(\alpha_j/2)
$, and $\hat{l}_j=|v_j-v_{j-1}|=2\eps_n\cos(\alpha_j/2)$,
we easily obtain by simple calculations that
\begin{align*}
&f_{\alpha_j}(0)=(\pi-\alpha_j)/2, \quad
f_{\alpha_j}'(0)=\hat{l}_j/(2l_j), \quad
f_{\alpha_j}''(0)= 0. 
\end{align*}
Taking into account that equation~\eqref{eqFgen} holds with $R_n\equiv \eps_n$
and using the uniform boundedness~\eqref{boundalpha} of the labelling~$\alpha$, 
inequality~\eqref{eqMonoq} yields
\begin{equation}\label{eqTaylorR}
0\leq \sum_{j=1}^m f_{\alpha_j}'(0)x_j + \Od(\eps_n^3).
\end{equation}
To evaluate this sum, expand
\begin{gather*}
h_n(z_j)-h_n(z)= \Re(\log g'(z_j)-\log g'(z)) = \Re(a(z_j-z)+ b(z_j-z)^2 )+
\Od(\eps_n^3) \\[0.5ex]
\begin{split} \text{and }\
x_j&= h_n(z_j)-h_n(z)-\beta|z_j|^2+\beta|z|^2 \\
&= \Re(a(z_j-z)+ b(z_j-z)^2 -2\beta \bar{z}(z_j-z)) - \beta l_j^2 +
\Od(\eps_n^3).
\end{split}
\end{gather*}
Noting that $f_{\alpha_j}'(0)(z_j-z) =(v_j-v_{j-1})/(2i)$ we get
\begin{align*}
\sum_{j=1}^m f_{\alpha_j}'(0)x_j &= \Re\Biggl(\frac{a-2\beta
  \bar{z}}{2i} \underbrace{\sum_{j=1}^m (v_j-v_{j-1})}_{=0} +
\frac{b}{2i} \underbrace{\sum_{j=1}^m
\underbrace{(v_j-v_{j-1}) (z_j-z)}_{
\stackrel{\eqref{eqzva}}{=} (a_j-a_{j-1})(a_j+a_{j-1})}}_{=0}\Biggr) \\
&\,\quad -\beta \sum_{j=1}^m\frac{l_j\hat{l}_j}{2}
+ \Od(\eps_n^3).
\end{align*}
Thus from inequality~\eqref{eqTaylorR}, remembering 
$l_j=2\eps_n\sin(\alpha_j/2)$ and $\hat{l}_j=2\eps_n\cos(\alpha_j/2)$,  we
arrive at
\begin{equation*}
0\leq -\beta \eps_n^2 \sum_{j=1}^m\sin(\alpha_j/2) \cos(\alpha_j/2) +
\Od(\eps_n^3) \quad\iff\quad \beta \sum_{j=1}^m \sin(\alpha_j) \leq
\Od(\eps_n).
\end{equation*}
Note that $\eps_n^2\sum_{j=1}^m \sin(\alpha_j)>\pi \eps_n^2$ is the
area of the rhombic faces incident to the vertex $z$.
Thus we conclude that $\beta =\Od(\eps_n)$.
This means, that if we choose $\beta= C\eps_n$ with $C>0$ a sufficiently
large constant and if $\eps_n$ is small enough such that $C\eps_n< 1$,
then $p$ will have no maximum in $V_{int}(G_n)$. In
that case, as we have $p(z)=\beta |z|^2=\Od(\eps_n)$ on
$V_{\partial}(G_n)$, we deduce that $p(z)\leq \Od(\eps_n)$ in
$V(G_n)$ and thus
\begin{equation}
  t_n(z)-h_n(z)\leq \Od(\eps_n) \qquad \text{for } z\in V(G_n).
\end{equation}
\hspace{\parindent}
The proof for the reverse inequality is almost the same. The only
modifications needed are reversing the sign of $\beta$ and a few
inequalities.
\end{proof}

\begin{remark}\label{remeps2}
The statement of Lemma~\ref{lem1} can be improved to
\begin{equation}\label{eqlem1eps2}
h_n(z)-t_n(z)=\Od(\eps_n^2)
\end{equation}
in the case of a 'very regular' isoradial circle pattern. These are isoradial
circle patterns such that for every oriented edge $e_{j_1}=z_{j_1}-z\in
\vec{E}(G)$ incident to an
interior vertex $z\in V_{int}(G)$ there is
another parallel edge $e_{j_2}=z_{j_2}-z \in \vec{E}(G)$ with opposite direction
incident to $z$, that is $e_{j_2}=- e_{j_1}$.
Furthermore, the corresponding intersection angles agree:
$\alpha([z,z_{j_1}]) =\alpha([z,z_{j_2}])$.
This additional regularity property holds for example for an orthogonal circle
pattern with the combinatorics of a part of the square grid, see
Figure~\ref{figExQuad}.

The proof of estimation~\eqref{eqlem1eps2} follows the same reasonings as above,
but makes an $\Od(\eps_n^4)$-analysis.
The additional regularity implies
that all terms of order $\eps_n^3$ vanish.
\end{remark}

\begin{definition}\label{defLap}
 For a function $\eta:V(G)\to\R$ define a discrete Laplacian by
\begin{equation}\label{eqLap}
  \Delta \eta (z)=\sum_{[z,z_j]\in E(G)} 2 f'_{\alpha([z,z_j])}(0)
(\eta(z_j)-\eta(z)).
\end{equation}
\end{definition}

As $f'_{\alpha([z_1,z_2])}(0)>0$ one immediately has the following
\begin{lemma}[Maximum Principle]\label{MaxPrinzip}
If $ \Delta \eta \geq 0$ on $V_{int}(G)$
then the maximum of $\eta$ is attained at the boundary $V_\partial (G)$.
\end{lemma}

The proof of Lemma~\ref{lem1} actually shows, that 
$t_n-h_n$ is almost harmonic. More precisely, we have
$\Delta(t_n-h_n)= \Od(\eps_n^3)$. Adding a suitable subharmonic
function $\beta |z|^2$ with $\beta>0$ big enough, we deduce that the resulting
function $p$ is {\em subharmonic}, that is $\Delta p \geq 0$, such
that $p$ attains its maximum at the boundary. This is also
important for our proof of the following lemma.

\begin{lemma}\label{lem3}
Let $t_n$ and $h_n$ be defined as in Lemma~\ref{lem1}.
Let $K\subset D$ be a compact subset in $D$. Then the following
estimation holds for every interior vertex $z\in
V_{int}(G_n)\cap K$ such that all its incident vertices $z_1,\dots,z_l$ are
also in $V_{int}(G_n)\cap K$:
\begin{equation}\label{eqestdiff}
t_n(z_j)-h_n(z_j)-(t_n(z)-h_n(z))=\Od(\eps_n(-\log\eps_n)^{-\frac{1}{2}}).
\end{equation}
for $j=1,\dots,l$.
The constant in the $\Od$-notation may depend on $K$, but not on $n$ or $z$.
\end{lemma}

The proof of Lemma~\ref{lem3} uses the following estimation for
superharmonic functions, which is a version of Corollary~3.1 of~\cite{SC97};
see also~\cite[Remark~3.2 and Lemma~2.1]{SC97}.

\begin{proposition}[\cite{SC97}]\label{propSC}
Let $G$ be an undirected connected graph without loops and let
$c:E(G)\to \R^+$ be 
a positive weight function on the edges. Denote $c(e)=c(x,y)$
for an edge $e=[x,y]\in E(G)$ and assume that 
\[m= \max_{[x,y]\in E(G)}
\sum_{[x,z]\in E(G)}\frac{c(x,z)}{c(x,y)}<\infty.\]
Denote by $d(x,y)$ the combinatorial distance
between two vertices $x,y\in V(G)$ in the graph $G$. Let $B_x(\varrho)= \{y\in
V(G) :
d(x,y)\leq \varrho\}$ be the combinatorial ball of radius $\varrho>0$ around the
vertex
$x\in V(G)$. Fix $x\in V(G)$ and $R\geq 4$ and set
\[A_R=\sup_{1\leq \varrho\leq R} \varrho^{-2}W_x(\varrho),\qquad
\text{where}\qquad W_x(\varrho)= \sum_{\substack{z\in B_x(\varrho),\ y\in V(G)\\
d(x,z)<d(x,y)}} c(z,y).\] 

Let $u$ be a positive superharmonic function in $B_x(R+1)$, that is
\[\sum_{[z,w]\in E(G)} c(z,w)(u(z)-u(w)) \leq 0\] 
for all $w\in B_x(R+1)$. Let $y$ be incident to $x$ in $G$. Then
\begin{equation*}
\left| \frac{u(x)}{u(y)}-1\right| \leq
\frac{4m^2\sqrt{A_R}}{\sqrt{c(x,y)\log_2 R}}.
\end{equation*}
\end{proposition}

\begin{proof}[Proof of Lemma~\ref{lem3}.]
Lemma~\ref{lem1} implies that $t_n(z_j)-t_n(z)=\Od(\eps_n)$
for all incident vertices $z,z_j\in V(G_n)$ since $h_n=\log|g'|$ is a
$C^\infty$-function. Consider a Taylor expansion about $0$ of
\[0=\Biggl(\sum_{j=1}^m f_{\alpha(z,z_j)}(t_n(z_j)-t_n(z))\Biggr)
-\pi.\] 
Similar reasonings as in the proof of Lemma~\ref{lem1} imply that
$\Delta t_n(z)=\Od(\eps_n^3)$.
Let $p=t_n-h_n+\beta|z|^2$ with $\beta\in(0,1)$.
Choosing $\beta= C\eps_n$ with a sufficiently
large constant $C>0$ and $\eps_n$ small enough we deduce
similarly as in the proof of Lemma~\ref{lem1} that
$\Delta p(z)\geq 0$ for all interior vertices $z\in V_{int}(G_n)$.
Now define the positive function
$ \hat{p}=\eps_n + \|p\| -p$.
Then $ \Delta\hat{p}(z)\leq 0$ for all $z\in V_{int}(G_n)$. The proof
of Lemma~\ref{lem1} shows that there is a constant $C_1$, depending
only on $g$, $D$, and the labelling $\alpha$, such that
$\|p\|\leq C_1\eps_n$. Thus $\|\hat{p}\|\leq C_2 \eps_n$ with
$C_2=2C_1+1$. 

To finish to proof, we apply Proposition~\ref{propSC} to the
superharmonic function $\hat{p}$. Remember that $G_n$ is a connected
graph without loops and $c(e):=2f_{\alpha(e)}'(0)>0$ defines a positive
weight function on the edges. The bound~\eqref{boundalpha} on the
labelling $\alpha$ implies that 
\[m= \max_{[x,y]\in E(G_n)}
\sum_{[x,z]\in E(G_n)} \frac{c(x,z)}{c(x,y)}<
\frac{2\pi}{\pi/2-C}\frac{\cot(\pi/4-C/2)}{\cot(\pi/4+C/2)}=:C_3 <\infty.\]
Let $x\in V_{int}(G_n)$.
Note that
\[W_x(\varrho)= \sum_{\substack{z\in B_x(\varrho),\; y\in V(G_n)\\
    d(x,z)<d(x,y)}} c(z,y) \leq \left(\max_{e\in E(G_n)} c(e)\right)
|F_w(x,\varrho)|,\]
where $F_w(x,\varrho)$ is the set of all faces of ${\mathscr D}_n$ with one
white
vertex $z\in B_x(\varrho)$ and $|F_w(x,\varrho)|$ denotes the number of faces of
$F_w(x,\varrho)$.
Now, $\max_{e\in E(G_n)} c(e)<\cot(\pi/4-C/2)/2$ and 
\[F_w(x,\varrho) \subset
D_x((\varrho+1)2\eps_n)= \{w\in\C : |w-x|\leq (\varrho+1)2\eps_n\},\]
as the edge lengths in $G_n$ are smaller than $2\eps_n$. Remember that
$F(f)=\eps_n^2 \sin\alpha(f) > \eps_n^2\sin(\pi/2-C)$ is the area of the face
$f\in F({\mathscr D}_n)$. Thus
\[|F_w(x,\varrho)| < \frac{\pi((\varrho+1)2\eps_n)^2}{\eps_n^2\sin(\pi/2-C)}
\leq
\frac{16\pi \varrho^2}{\sin(\pi/2-C)}=:\varrho^2 C_4 \]
for all $\varrho\geq 1$. Therefore we obtain $A_R=\sup_{1\leq \varrho\leq R}
\varrho^{-2}W_x(\varrho) <C_4$, where the upper bound $C_4$ is
independent of $R\geq 4$ and $n\in\N$.

Let $K\subset D$ be compact. 
Denote by $\mathbb{e}$ the Euclidean 
distance (between a point and a compact set or between closed sets of
$\R^2\cong \C$).
Let $z\in V(G_n)\cap K$ and set
$(R+1)=d(z,V_\partial(G_n))$ to be the combinatorial distance
from $z$ to the boundary of $G_n$. Let $z_j\in V(G_n)$ be incident to $z$.
As the labelling $\alpha$ is bounded, $\eps_n\to 0$, and ${\mathscr D}_n$
approximates $D$, we deduce that
$R\geq \eps_n^{-1} C_5 \geq 4$ if $n\geq n_0$ is large enough. Thus
for all $n\geq n_0$ and all $z\in V(G_n)\cap K$
\[
1/\sqrt{\log_2 R}\leq 1/\sqrt{\log_2 C_5 -\log_2\eps_n}\leq
\sqrt{2}/\sqrt{-\log_2\eps_n}
\]
holds  by our
assumptions. Proposition~\ref{propSC} implies that
\[\left| \frac{\hat{p}(z)}{\hat{p}(z_j)}-1\right| \leq
\frac{4C_3^2\sqrt{C_4}\sqrt{2}}{\sqrt{-c(z,z_j)\log_2\eps_n }}\]
for all incident vertices $z,z_j\in V(G_n)\cap K$ and $n\geq n_0$. As
$c(z,z_j)\leq \cot(\pi/4-C/2)/2$ and $\|\hat{p}\|\leq C_2\eps_n$ we
finally arrive at the desired estimation
\[|t_n(z_j)-h_n(z_j)-(t_n(z)-h_n(z))| =|\hat{p}(z)-\hat{p}(z_j)| \leq
C_6 \eps_n(-\log_2\eps_n)^{-\frac{1}{2}}\]
for all incident vertices $z,z_j\in V(G_n)\cap K$ and $n\geq n_0$, where the
constant $C_6$ depends on $C_2,\dots,C_5$, that
is only on $g$, $D$, $C$, $\hat{C}$, and $K$.
\end{proof}

\begin{lemma}\label{lem2}
  Let $\vec{e}=\overrightarrow{uv}\in \vec{E}({\mathscr D}_n)$ be a directed
 edge with $u\in V(G_n)$ and $v\in V(G_n^*)$. Denote by
  $\delta_n({e})$  the combinatorial 
  distance in ${\mathscr D}_n$ from $e=[u,v]$ to $[z_0,v_0]$,
  that is the least integer $k$ such that there is a sequence of edges
  $\{[z_0,v_0]=e_1,e_2,\dots,e_k=e\}\subset E({\mathscr D}_n)$ such that
 the edges $e_{m+1}$ and $e_m$ are incident to the same face in ${\mathscr
   D}_n$ for $m=1,\dots,k-1$. Then  
\begin{equation}\label{eqvarphigen}
  \varphi_n(\vec{e})=\arg g'(v) +\phi_n(\vec{e}) +
  \delta_n({e})\Od(\eps_n(-\log_2\eps_n)^{-\frac{1}{2}}).
\end{equation}
The constant in the notation $\Od(\eps_n(-\log_2\eps_n)^{-\frac{1}{2}})$
may depend on the distance of $v$ to the boundary $\partial D$.
\end{lemma}
Note that if $\partial D$ is smooth, then
$\delta_n({e})=\Od(\eps_n^{-1})$. In
general we have $\delta_n({e})=\Od(\eps_n^{-1})$ on compact subsets $K\subset
D$, where the constant in the notation $\Od(\eps_n^{-1})$ may depend on
$K$. In any case, on compact subsets of $D$ we have
\begin{equation}\label{eqvarphi}
  \varphi_n(\vec{e})=\arg g'(v) +\phi_n(\vec{e}) +
  \Od((-\log_2\eps_n)^{-\frac{1}{2}}).
\end{equation}

\begin{proof}
Using the notation of Figure~\ref{anglesFig} (left),
equation~\eqref{eqestdiff} implies 
\begin{multline*}
f_{\alpha}(\log r_n(z_+) - \log r_n(z_-))
= f_{\alpha}(t_n(z_+) -t_n(z_-)) \\
= f_{\alpha}(0) + f_{\alpha}'(0)(\log
|g'(z_+)|-\log |g'(z_-)|) 
+\Od(\eps_n(-\log_2\eps_n)^{-\frac{1}{2}}).
\end{multline*}
As in Lemma~\ref{lem1} we have
$2f_{\alpha}'(0)= \frac{|v_+-v_-|}{|z_+-z_-|}= \frac{v_+-v_-}{i(z_+-z_-)}$ with
the same notation which yields
\begin{multline*}
2f_{\alpha}'(0)(\log |g'(z_+)|-\log |g'(z_-)|)\\
= \frac{v_+-v_-}{i(z_+-z_-)}\Re(a(z_+-z_-)) +
\Od(\eps_n(-\log_2\eps_n)^{-\frac{1}{2}}) \\
= \Im(a(v_+-v_-)) +\Od(\eps_n(-\log_2\eps_n)^{-\frac{1}{2}}) \\
= \arg g'(v_+)-\arg g'(v_-)+\Od(\eps_n(-\log_2\eps_n)^{-\frac{1}{2}}),
\end{multline*}
where $a=\frac{g''((z_++z_-)/2)}{g'((z_++z_-)/2)}
=\frac{g''((v_++v_-)/2)}{g'((v_++v_-)/2)}$.

By Lemma~\ref{MonoCond} we can choose the angle functions
$\phi_n$ and $\varphi_n$ on any minimal sequence of edges
$\{[z_0,v_0]=e_1,e_2,\dots,e_k=e\}\subset E({\mathscr D}_n)$ such that
equations~\eqref{eqCR1}--\eqref{eqCR6} are satisfied without the
$\!\!\pmod{2\pi}$-term. Using the above considerations of
$2f_{\alpha}(\log r_n(z_+) - \log r_n(z_-))$ and the normalization of
$\varphi_n$, we arrive at equation~\eqref{eqvarphigen}.
\end{proof}

\begin{proof}[Proof of Theorem~\ref{theoConvC1}]\label{proofConvC1beg}
Consider a compact subset $K$ of $D$. 
Let $z\in V(G_n)\cap K$ and $v\in V(G^*_n)\cap K$ be vertices which
are incident in ${\mathscr D}_n$, that is $[z,v]\in E({\mathscr D}_n)$. Then
Lemmas~\ref{lem1} and~\ref{lem2} imply that
\begin{align*}
\log g'(z) &= \log|g'(z)| +i \arg g'(z)  \\
&= \log(r_n(z)/R_n(z)) + i(\varphi_n(\overrightarrow{zv})
-\phi_n(\overrightarrow{zv})) +\Od((-\log_2\eps_n)^{-\frac{1}{2}}).
\end{align*}
As $g'$ and thus the quotient $r_n/R_n$ is uniformly bounded, we obtain
\begin{equation}
g'(z)=
\underbrace{\frac{r_n(z)}{R_n(z)}e^{i(\varphi_n(\overrightarrow{zv})
  -\phi_n(\overrightarrow{zv}))}}_{=q_n(z)}
+\Od((-\log_2\eps_n)^{-\frac{1}{2}}).
\end{equation}
This implies the uniform convergence on compact subsets of $D$ of
$q_n$ to $g'$.

Convergence of $g_n$ is now proven by using suitable integrations of
$g'$ and $q_n$.
Let $w\in V(G_n^*)$ and consider a shortest path
$\gamma$ in ${G}_n^*$ from $v_0$ to $w$ with vertices
$\{v_0=w_1,w_2,\dots,w_k=w\}\subset V({G}_n^*)$. Then 
\begin{align*}
 g(w) &=g(v_0)+ \int_{\gamma} g'(\zeta)d\zeta 
 =g(v_0)+ \sum_{j=1}^{k-1}g'(w_{j+1})(w_{j+1}-w_j) +\Od(\eps_n) \\
&= g(v_0)+ \sum_{j=1}^{k-1}q_n(w_{j+1})(w_{j+1}-w_j)
+\Od((-\log_2\eps_n)^{-\frac{1}{2}}),
\end{align*}
because $g'(w_j)-q_n(w_j)=\Od((-\log_2\eps_n)^{-\frac{1}{2}})$,
$w_{j+1}-w_j=\Od(\eps_n)$ and $k=\Od(\eps_n^{-1})$ on
compact sets. Thus it only remains to show that
\begin{equation}\label{eqintg}
  p_n(w)=g(v_0)+ \sum_{j=1}^{k-1} q_n(w_{j+1})(w_{j+1}-w_j)
  +\Od((-\log_2\eps_n)^{-\frac{1}{2}}).
\end{equation}
Remembering
\begin{align*}
q_n(v_+) &= \frac{r_n(z_+)}{R_n(z_+)}
\text{e}^{i(\varphi_n(\overrightarrow{w_{j+1}z_+})   
  -\phi_n(\overrightarrow{w_{j+1}z_+}))} +
\Od((-\log_2\eps_n)^{-\frac{1}{2}}), \\
(w_{j+1}-w_j) &=2R_n(z_+)\cos(\alpha([w_{j+1},w_j])/2)
\text{e}^{i(\phi_n(\overrightarrow{w_{j+1}z_+}) 
  -(\pi/2-\alpha([w_{j+1},w_j])/2))},
\end{align*}
we can conclude that
\begin{align*}
q_n(w_{j+1})(w_{j+1}-w_j)
&= p_n(w_{j+1})-p_n(w_j)+\Od(\eps_n(-\log_2\eps_n)^{-\frac{1}{2}}),
\end{align*}
where $z_-,z_+\in V(G)$ are incident to $w_{j+1}$ and $w_j$ and we
have used the
notations are as in Figure~\ref{anglesFig} (left) with $w_j=v_-$ and
$w_{j+1}=v_+$. As we have normalized $g(v_0)=p(v_0)$, this proves
equation~\eqref{eqintg} and therefore the uniform convergence of $p_n$
to $g$ on compact subsets of $D$.
\end{proof}\label{proofConvC1end}

\begin{remark}\label{remtheo1}
Theorem~\ref{theoConvC1} may easily be generalized in the following 
ways. First, we may consider 'nearly isoradial' circle
patterns which satisfy $R_n(z)=\Od(\eps_n)$ for all vertices $z\in V(G_n)$
and $R_n(z_1)/R_n(z_2)=1+\Od(\eps_n^3)$ for all edges $[z_1,z_2]\in E(G_n)$.

Second, we may omit the assumption that the whole domain $D$ is approximated by
the rhombic embeddings ${\mathscr D}_n$. Then the convergence claims remain
true for compact subsets of any open domain $D'\subset D$ which is covered or
approximated by the rhombic embeddings and contains $v_0$.
\end{remark}

Using the angle function instead of the radius function, we obtain the 
following analog of
Theorem~\ref{theoConvC1} for Neumann boundary conditions.

\begin{theorem}\label{theoConvC1Neumann}
Under the same assumptions as in Theorem~\ref{theoConvC1} and with the same
notation, assume further
that $\eps_n$ is sufficiently small such that for all $n\in\N$
\begin{equation}\label{eqcondg'}
\sup_{v\in D} \max_{\theta\in [0,2\pi]}|\arg
g'(v+2\eps_n\text{e}^{i\theta}) -\arg g'(v)| <\frac{\pi}{2}-C< \min_{e\in
  E(G_n)}(\pi-\alpha(e)).
\end{equation}

Define an angle function on the oriented boundary edges by
\begin{equation*}
 \varphi_n(\vec{e})=\phi_n(\vec{e}) +\arg g'(v),
\end{equation*}
where $\vec{e}=\overrightarrow{zv}\in\vec{E}({\mathscr D}_n)$ and $v\in
V(G_n^*)$.
Then there is a circle pattern ${\mathscr C}_n$
for $G_n$ and $\alpha_n$ with radius 
function $r_n$ and angle function $\varphi_n$ with these boundary values.

Suppose that this circle pattern is normalized such that
\[r_n(z_0)=R_n(z_0)|g'(z_0)|, \]
where $z_0\in V(G_n)$ is chosen such that the disk bounded by the circle
$C_{z_0}$ contains $0$.
Suppose further that ${\mathscr C}_n$ is normalized  by a translation such that
\begin{equation}
  p_n(v_0)=g(v_0),
\end{equation}
where $p_n(v)$ denotes the intersection point corresponding to $v\in
V(G_n^*)$. 
 
Then $q_n\to g'$ and $g_n\to g$ uniformly on compact
subsets in $D$ as $n\to\infty$.
\end{theorem}

\begin{proof}
The existence claim for the circle pattern with Neumann
boundary conditions follows from~\cite[Theorem 3]{BS02} using the
assumption~\eqref{eqcondg'}.

Theorem~\ref{theoCompare} shows that the difference $\varphi_n -\phi_n$ gives
rise to a function $\delta_n:V(G_n^*)\to\R$ with boundary values given by $\arg
g'$.
The proof of the convergence claim is similar to the proof of
Theorem~\ref{theoConvC1}. The roles of $\delta_n= \varphi_n-\phi_n$ and
$\log(r_n/R_n)$
have to be interchanged in Lemmas~\ref{lem1}, \ref{lem3},
and~\ref{lem2} and similarly $\arg g'=\Im(\log g')$ has to be considered
instead of $\log|g'|=\Re(\log g')$. The role of equation~\eqref{eqFgen}
is substituted by equation~\eqref{eqintersectionPoint}.
\end{proof}

\section{Quasicrystallic circle patterns}\label{secQuasiCirc}

The order of convergence in Theorems~\ref{theoConvC1}
and~\ref{theoConvC1Neumann} can be improved for a special class of isoradial
circle patterns with a uniformly bounded number of different edge directions
and a local deformation property. In the following, we introduce suitable
terminology and some useful results.

As the kites of an isoradial circle pattern are in fact rhombi, an embedded
isoradial circle pattern leads to a {\em rhombic embedding in $\C$} of the
corresponding b-quad-graph ${\mathscr D}$. Conversely, adding circles with
centers in the white vertices of a rhombic embedding and
radius equal to the edge length results in an embedded
isoradial circle pattern.

Given a rhombic embedding of a b-quad-graph ${\mathscr D}$, consider for each
directed edge $\vec{e}\in \vec{E}({\mathscr D})$ the vector of its embedding
as  a  complex number with length one. Half of the number of
different values of these directions is called the
{\em  dimension} $d$ of the rhombic embedding. If $d$ is finite, the  rhombic
embedding is called {\em quasicrystallic}.
A circle pattern for a b-quad-graph
${\mathscr D}$ is called a {\em quasicrystallic circle pattern} if there
exists a quasicrystallic rhombic embedding of ${\mathscr D}$ and if the
intersection angles are taken from this rhombic embedding.
The comparison function of the isoradial circle pattern ${\mathscr
C}_1$ for ${\mathscr D}$
and the quasicrystallic circle pattern ${\mathscr C}_2$ will also be
called {\em comparison function for ${\mathscr C}_2$}.

Quasicrystallic circle patterns were introduced in~\cite{BMS05}.
Certainly, this property only makes sense for infinite graphs or
infinite sequences of graphs with growing number of vertices and
edges.

In the following we will identify the b-quad-graph
${\mathscr D}$ with a rhombic embedding of ${\mathscr D}$.

\subsection{Quasicrystallic rhombic embeddings and $\Z^d$}\label{secQuasiZd}

\begin{floatingfigure}[l]{4.83cm}
\begin{center}
\includegraphics[height=4cm]{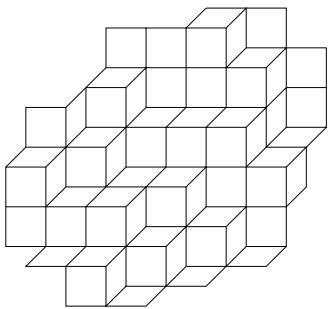}
\end{center}
\caption{An example of a combinatorial surface $\Omega_{\mathscr
    D}\subset \Z^3$.}\label{figExOmegaD}
 \end{floatingfigure}
Any rhombic embedding of a connected b-quad-graph ${\mathscr D}$ can be seen as
a sort of projection of a certain two-dimensional subcomplex (combinatorial
surface) $\Omega_{\mathscr D}$ of the multi-dimen\-sio\-nal lattice $\Z^d$ (or
of a multi-dimen\-sio\-nal lattice $\cal L$ which is
isomorphic to $\Z^d$). An illustrating example is given in 
Figure~\ref{figExOmegaD}.

The combinatorial surface
$\Omega_{\mathscr D}$ in $\Z^d$ can be constructed in the following way.
Denote the set of the different edge directions of $\mathscr D$ by ${\cal
A}=\{\pm
a_1,\dots,\pm a_d\}\subset \Sp^1$. We suppose that $d>1$ and that any two
non-opposite elements of ${\cal A}$ are linearly independent over $\R$.
Let ${\mathbf e}_1,\dots,{\mathbf e}_d$ denote the standard orthonormal basis of
$\R^d$. Fix a white vertex $x_0\in V({\mathscr D})$ and
the origin of $\R^d$. Add the edges of $\{\pm {\mathbf e}_1,\dots,\pm {\mathbf
e}_d\}$ at the origin which correspond to
the edges of $\{\pm a_1,\dots,\pm a_d\}$ incident to $x_0$ in
$\mathscr D$, together with their endpoints. Successively continue the
construction at the new endpoints. Also, add two-dimensional facets (faces) of
$\Z^d$ corresponding to faces of $\mathscr D$, spanned by incident
edges.

A combinatorial surface $\Omega_{\mathscr D}$ in $\Z^d$
corresponding to a quasicrystallic rhombic embedding can be
characterized using the following monotonicity
property, see~\cite[Section~6]{BMS05} for a proof.
\begin{lemma}[Monotonicity criterium]\label{lemMonoton}
Any two points of $\Omega_{\mathscr D}$ can be connected by a path in
$\Omega_{\mathscr D}$ with all
directed edges lying in one $d$-dimensional octant, that is all
directed edges of this path are elements of one of the $2^d$ subsets of $\{\pm
{\mathbf e}_1,\dots, \pm {\mathbf e}_d\}$ containing $d$ linearly independent
vectors. 
\end{lemma}

An important class of examples of rhombic embeddings of b-quadgraphs
can be constructed using ideas of the grid projection method
for quasiperiodic tilings of the plane; see for example~\cite{DK,GR,Se}.

\begin{example}[Quasicrystallic rhombic embedding obtained from a
plane]\label{exquasi}
 Let $E$ be a two-dimensional plane in $\R^d$ and $t\in E$.
Let ${\mathbf e}_1,\dots,{\mathbf e}_d$ be the standard orthonormal basis of
$\R^d$. We assume
that $E$ does not contain any of the segments $s_j=\{ {\mathbf t}
+\lambda{\mathbf e}_j : \lambda \in [0,1]\}$ for $j=1,\dots,d$.
Then we can choose positive
numbers $c_1,\dots, c_d$ such that the orthogonal projections
$P_E(c_j{\mathbf e}_j)$ have length 1.
(If $E$ contains two different segments $s_{j_1}$ and $s_{j_2}$, the following
construction only leads to the
standard square grid pattern $\Z^2$.  If $E$ contains exactly one such segment
$s_j$, then the construction may be adapted for the
remaining dimensions excluding ${\mathbf e}_j$.)
We further assume that the orthogonal projections
onto $E$ of the two-dimensional facets $E_{j_1,j_2}= \{\lambda_1
{\mathbf e}_{j_1}+\lambda_2{\mathbf e}_{j_2}: \lambda_1, \lambda_2\in[0,1]\}$
for $1\leq j_1<j_2\leq d$ are non-degenerate parallelograms.

Consider around each vertex ${\mathbf p}$ of the
lattice ${\cal L}=c_1\Z\times\dots \times c_d\Z$ the hypercuboid
$V=[-c_1/2,c_1/2]\times \dots \times [-c_d/2,c_d/2]$, that is the Voronoi
cell ${\mathbf p}+V$.
These translations of $V$ cover $\R^d$.
We build an infinite monotone two-dimensional surface $\Omega^{\cal L}(E)$ in
${\cal L}$ by the following construction.
 The basic idea is illustrated in Figure~\ref{figExquasi} (left).

If $E$ intersects the interior
of the Voronoi cell of a lattice point (i.e.\ $({\mathbf p}+V)^\circ\cap
E\not=
\emptyset$ for ${\mathbf p}\in {\cal L}$), then this point belongs to
$\Omega^{\cal L}(E)$. Undirected edges correspond to intersections
of $E$ with
the interior of a $(d-1)$-dimensional facet bounding two
Voronoi cells. Thus we get a connected graph
in ${\cal L}$. An intersection of $E$ with the interior
of a translated $(d-2)$-dimensional facet of $V$ corresponds to a
rectangular two-dimensional face of the lattice. By construction, the
orthogonal projection
of this graph onto $E$ results in a planar connected graph whose faces
are all of even degree ($=$ number of incident edges or of incident
vertices). A face of degree bigger than 4 corresponds to an
intersection of $E$ with the translation of a $(d-k)$-dimensional facet
of $V$ for some $k\geq 3$. Consider the vertices and edges of such a
face and the corresponding points and edges in the lattice ${\cal L}$. These
points lie on a combinatorial
$k$-dimensional hypercuboid contained in ${\cal L}$.
By construction, it is easy to see that 
there are two points of the $k$-dimensional
hypercuboid which are each incident to $k$ of the given vertices. We choose
a point with least distance from $E$ and add it to the
surface. Adding edges to neighboring vertices splits
the face of degree $2k$ into $k$ faces of degree 4.
Thus we obtain an infinite monotone two-dimensional combinatorial surface
$\Omega^{\cal L}(E)$ which projects to an infinite rhombic embedding covering
the whole plane $E$.
\end{example}

\begin{figure}[tb]
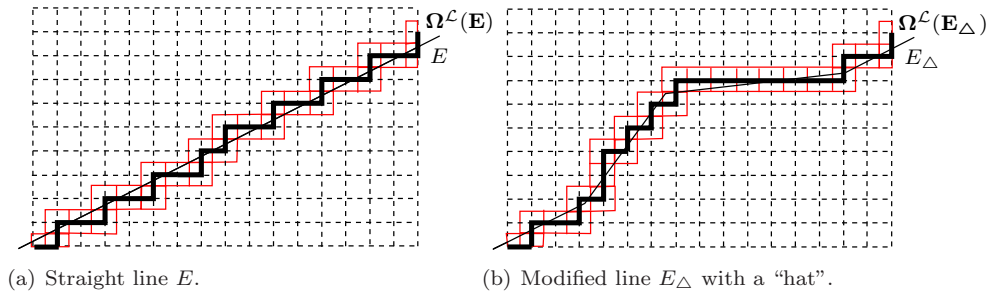

\begin{center}
\subfigure[Straight line $E$.]{
\input{ExampleOmegaEZ2.pstex_t}
}
\subfigure[Modified line $E_\triangle$ with a ``hat''.]{
\input{ExampleOmegaEtriangleZ2.pstex_t}
}
\end{center}
\hspace*{1ex}
\caption{Example for the usage of the grid-projection method in
  $\Z^2$. The Voronoi cells which contain $\Omega^{\cal L}$
are marked in red.}\label{figExquasi}
\end{figure}

\begin{example}[Modification of the construction in
Example~\ref{exquasi}]\label{exquasihat}
The method used in the preceding example can be modified to result in similar
but
different rhombic embeddings.  The basic idea
is illustrated in Figure~\ref{figExquasi} (right).

Let $E$ be
a two-dimensional plane in $\R^d$ satisfying the same assumptions as in
Example~\ref{exquasi}.
Let ${\mathbf N}\not=0$ be a vector orthogonal to $E$. Let $\triangle$ be an
equilateral triangle in $E$ with vertices ${\mathbf t}_1,{\mathbf
t}_2,{\mathbf
t}_3$ and let $\mathbf s$ be
the intersection point of the bisecting lines of the angles. Consider the
two-dimensional facets of the three-dimensional tetrahedron $T(
\triangle,{\mathbf N})$ spanned by
the four vertices ${\mathbf t}_1,{\mathbf t}_2,{\mathbf t}_3,{\mathbf s}
+{\mathbf N}$. Exactly one of these facets is completely
contained in $E$ (this is the triangle $\triangle$). We remove the
triangle from $E$ and add instead the remaining facets of
$T(\triangle,{\mathbf N})$. Let $E_\triangle$
be the resulting two-dimensional surface. Note that $E_\triangle$ is
orientable like $E$.
Define $\gamma\in(0,\pi/2)$ by $\gamma= \arctan(
2\sqrt{3}\|{\mathbf N}\|/\|{\mathbf t}_1-{\mathbf t}_2\|)$, where $\|\cdot\|$
denotes the Euclidean norm of vectors in $\R^d$. $\gamma$ is the acute
angle between $E$ and the two-dimensional facets of 
$T(\triangle,{\mathbf N})$ not contained in $E$.

If $\gamma$ is small enough, then our
assumptions imply that we can apply the same construction algorithm to
$E_\triangle$ as for the plane $E$ in the previous example and obtain a
monotone surface $\Omega^{\cal L}(E_\triangle)$.
The orthogonal
projection of $\Omega^{\cal L}(E_\triangle)$ onto $E$ is a rhombic embedding
which coincides with the rhombic embedding from $\Omega^{\cal L}(E)$ except
for a finite part.
\end{example}

The quasicrystallic rhombic embedding of Example~\ref{exquasihat} may also
obtained using the following general concept of
local changes of rhombic embeddings.

\begin{definition}\label{defFlip}
Let ${\mathscr D}$ be a rhombic embedding of a finite simply connected
b-quad-graph with
corresponding combinatorial surface $\Omega_{\mathscr D}$ in $\Z^d$.
\begin{floatingfigure}[p]{5cm}
 \begin{center}
 \includegraphics[height=1.5cm]{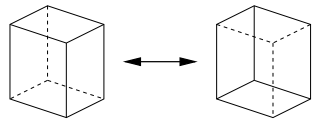}
 \end{center}
 \caption{A flip of a three-dimensional cube. The dashed edges and their
incident faces are not part of
the surface in $\Z^d$.}\label{figflip}
\end{floatingfigure}
Let $\hat{\mathbf z}\in V_{int}(\Omega_{\mathscr D})$ be an interior vertex
with exactly three incident two-dimensional facets of $\Omega_{\mathscr D}$.
Consider the
three-dimen\-sio\-nal cube with these boundary facets. Replace 
the three given facets with the three other two-dimensional facets of
this cube. This procedure is called a {\em flip}; see Figure~\ref{figflip} for
an illustration.

A vertex ${\mathbf z}\in\Z^d$ {\em can be  reached with
flips from $\Omega_{\mathscr D}$} if ${\mathbf z}$ is contained in a
combinatorial surface obtained
from $\Omega_{\mathscr D}$ by a suitable sequence of flips.
The set of all
vertices which can be reached with flips (including $V(\Omega_{\mathscr D})$)
will be denoted by ${\cal F}(\Omega_{\mathscr D})$.
\end{definition}

For further use we enlarge the set ${\cal F}(\Omega_{\mathscr D})$ of vertices
which can
be reached by flips from $\Omega_{\mathscr D}$ in the following way.
For a set of vertices $W\subset V(\Z^d)$ denote by $W^{[1]}$ the set $W$
together with all vertices incident to a two-dimensional facet of $\Z^d$ where
three of its four vertices belong to $W$. Define $W^{[k+1]}=(W^{[k]})^{[1]}$
inductively for all $k\in\N$. In particular, we denote for some arbitrary, but
fixed $\kappa\in\N$
\[{\cal F}_\kappa(\Omega_{\mathscr D})=({\cal F}(\Omega_{\mathscr
D}))^{[\kappa]}.\]

\subsection{Quasicrystallic circle patterns and integrability}\label{secCircInt}

Let $\mathscr D$ be a quasicrystallic rhombic embedding of a b-quad-graph.
The combinatorial surface $\Omega_{\mathscr D}$ in $\Z^d$ is
important by its connection with integrability. See also~\cite{BS08} for a
more detailed presentation and a deepened study of integrability and
consistency.

In particular, a
function defined on the vertices of $\Omega_{\mathscr D}$ which satisfies some
3D-consistent equation on all faces of $\Omega_{\mathscr D}$
can uniquely be extended to the {\em brick}
\begin{equation*}
\Pi(\Omega_{\mathscr D}) :=\{{\mathbf n}=(n_1,\dots, 
n_d)\in \Z^d:a_k(\Omega_{\mathscr D})\leq n_k\leq b_k(\Omega_{\mathscr D}),\
k=1,\dots, d\},
\end{equation*}
where $a_k(\Omega_{\mathscr D})= \min_{{\mathbf n}\in V(\Omega_{\mathscr
D})}n_k$ and $b_k(\Omega_{\mathscr D})= \max_{{\mathbf n}\in V(\Omega_{\mathscr
D})}n_k$.
Note that $\Pi(\Omega_{\mathscr D})$ is the hull of
$\Omega_{\mathscr D}$. A proof may be found in~\cite[Section~6]{BMS05}.
This extension of a function using a 3D-consistent equation will now be applied
for the comparison function $w$ defined in~\eqref{eqdefw} of two circle
patterns. In particular, we take for ${\mathscr C}_1$ the isoradial circle
pattern which corresponds to the quasicrystallic rhombic embedding $\mathscr
D$. Given another circle pattern ${\mathscr C}_2$ for $\mathscr D$ with the
same intersection angles, let $w$ be the comparison function for
${\mathscr C}_2$. Note that the Hirota Equation~\eqref{eqw} is
3D-consistent; see Sections~10 and~11 of~\cite{BMS05} for more details.
Thus $w$ considered as a function on
$V(\Omega_{\mathscr D})$ can uniquely be extended to the brick
$\Pi(\Omega_{\mathscr D})$ such that equation~\eqref{eqw} 
holds on all two-dimensional facets.
Additionally, $w$ and its extension are real valued on
white points $V_w(\Omega_{\mathscr D})$ and has value in $\Sp^1$ for black
points $V_b(\Omega_{\mathscr D})$.
This can easily be deduced from the Hirota Equation~\eqref{eqw}.

The extension of $w$ can be used to define a radius function for any rhombic
embedding with the same boundary faces as $\mathscr D$.

\begin{lemma}\label{lem2patt}
Let ${\mathscr D}$ and ${\mathscr D}'$ be two simply connected finite
rhombic embeddings of b-quad-graphs with the same edge directions. Assume that
${\mathscr D}$ and ${\mathscr D}'$ agree on all boundary faces. Let
${\mathscr C}$ be an (embedded) planar circle pattern for ${\mathscr D}$
and the labelling given by the rhombic embedding. Then there is an
(embedded) planar circle
pattern ${\mathscr C}'$ for ${\mathscr D}'$ which agrees with
${\mathscr C}$ for all boundary circles.
\end{lemma}
\begin{proof}
Consider the monotone combinatorial surfaces $\Omega_{\mathscr D}$ and
$\Omega_{\mathscr D}'$. Without loss of generality, we can assume that
$\Omega_{\mathscr D}$ and $\Omega_{\mathscr D}'$ have the same
boundary faces in $\Z^d$. Thus they both define the same brick
$\Pi(\Omega_{\mathscr D}) =\Pi(\Omega_{\mathscr D}')=:\Pi$. Given the
circle pattern ${\mathscr C}$, define the function $w$ on
$V(\Omega_{\mathscr D})$ by~\eqref{eqdefw}. Extend $w$ to the brick
$\Pi$ such that condition~\eqref{eqw} holds for all two-dimensional
facets. Consider $w$ on $\Omega_{\mathscr D}'$ and build the
corresponding pattern ${\mathscr C}'$, such that the points on the boundary
agree with
those of the given circle pattern ${\mathscr C}$. Equation~\eqref{eqw}
guarantees that all rhombi of $\Omega_{\mathscr D}'$
are mapped to closed kites. Due to the combinatorics, the chain of
kites is closed around each vertex. Since the boundary kites of ${\mathscr C}'$
are given by
${\mathscr C}$ which is an immersed circle pattern,
at every interior white point the angles of 
the kites sum up to $2\pi$. Thus ${\mathscr C}'$ is an immersed circle pattern.

Furthermore, ${\mathscr C}'$ is embedded if ${\mathscr C}$ is, because
${\mathscr C}'$ is an immersed circle pattern and ${\mathscr C}'$ and ${\mathscr
C}$ agree for all boundary kites.
\end{proof}

\section{$C^\infty$-convergence for quasicrystallic circle
 patterns}\label{secConvQuasiCinfty}

In order to improve the order of convergence in Theorem~\ref{theoConvC1}
we study partial derivatives of the extended radius function
using the integrability
of the Hirota equation~\eqref{eqw} and a Regularity Lemma~\ref{lemRegIso1}.

The following constants are useful
to estimate the possible orders of partial derivatives for a function
defined on ${\cal F}_\kappa(\Omega_{\mathscr D})$.
Note that ${\cal F}_\kappa(\Omega_{\mathscr D})\subset \Pi(\Omega_{\mathscr
D})$.

\begin{definition}\label{defC12}
Let ${\mathscr D}$ be a rhombic embedding of a finite simply connected
b-quad-graph with
corresponding combinatorial surface $\Omega_{\mathscr D}$ in $\Z^d$.
Let $J\subset \{1,\dots,d\}$ contain at least two different indices.

For $B\geq 0$ define a {\em combinatorial ball} of
radius $B$ about ${\mathbf z}\in V(\Z^d)$ using the
directions $\{\mathbf{e}_{j}: j\in J\}$ by
\begin{equation}
 U_{J}({\mathbf z},B)=\{{\mathbf \zeta}={\mathbf z}
+\sum_{j\in J}n_{j}\mathbf{e}_{j} : \sum_{j\in J}|n_{j}|\leq B\}. 
\end{equation}
The {\em radius of the largest ball} about ${\mathbf z}$ using these
directions
which is contained in ${\cal F}_\kappa(\Omega_{\mathscr D})$ is denoted by
$  B_{J}({\mathbf z},{\cal F}_\kappa(\Omega_{\mathscr D}))
=\max\{B\in\N : U_{J}({\mathbf z},B)\subset {\cal
F}_\kappa(\Omega_{\mathscr D})\}$.

Denote by
$d(\hat{\mathbf z},\partial \Omega_{\mathscr D})$ the {\em combinatorial
distance} of
$\hat{\mathbf z}\in V(\Omega_{\mathscr D})$ to the boundary $\partial
\Omega_{\mathscr D}$, that is the
smallest integer $K$ such that there is a connected path with $K$ edges
contained in $\Omega_{\mathscr D}$ from
$\hat{\mathbf z}$ to a boundary vertex of $\partial \Omega_{\mathscr D}$.
For further use, we define the constant
\begin{equation}
  C_{J}({\cal F}_\kappa(\Omega_{\mathscr D}))
=\min\left\{ \frac{B_{J}(\hat{\mathbf z},{\cal F}_\kappa(\Omega_{\mathscr D}))
+1}{d(\hat{\mathbf z},\partial \Omega_{\mathscr D})} : \hat{\mathbf z}\in
V_{int}(\Omega_{\mathscr D}) \right\} >0.
\end{equation}
\end{definition}
Note as an immediate consequence that for all $\hat{\mathbf z}\in
V_{int}(\Omega_{\mathscr D})$ 
\[ U_{J}(\hat{\mathbf z}, \lceil C_{J}({\cal
F}_\kappa(\Omega_{\mathscr
D})) d(\hat{\mathbf z},\partial \Omega_{\mathscr D})-1\rceil )\subset
{\cal F}_\kappa(\Omega_{\mathscr D}),\]
where $\lceil s \rceil$ denotes the smallest integer bigger than $s\in\R$.

The following theorem is an improved version of Theorem~\ref{theoConvC1} for
a specified class of quasicrystallic circle patterns.

\begin{theorem}\label{theoConvQCinfty}
Under the assumptions of Theorem~\ref{theoConvC1} and with the same notation,
let $d\in\N$ with $d\geq 2$ be a constant and assume further that
${\mathscr D}_n$ is a quasicrystallic rhombic embedding
in $D$ with edge lengths $\eps_n$ and dimension
$d_n\leq d$. 
The directions of the edges 
are elements of the set $\{\pm a_1^{(n)},\dots,\pm a_{d_n}^{(n)}\}\subset
\Sp^1$ such that any two of the vectors of
$\{a_1^{(n)},\dots,a_{d_n}^{(n)}\}$ are linearly independent. The
possible angles
are uniformly bounded, that is for all $n\in\N$ the scalar product
is strictly bounded away from $1$,
\begin{equation}\label{eqproda}
|\langle a_i^{(n)},a_j^{(n)}\rangle |\leq \cos(\pi/2 +C)<1,
\end{equation}
for all $1\leq i<j\leq d_n$ and some constant $0<C<\pi/2$. Consequently, the
intersection angles $\alpha_n$ are uniformly bounded in the sense that
for all $n\in\N$ and all faces $f\in F({\mathscr D}_n)$ there holds
\begin{equation}\label{boundalphaCinf}
|\alpha_n(f)-\pi/2| \leq C.
\end{equation}

Let $\kappa\in\N$, let $J_0\subset\{1,\dots,d\}$ contain at least two
indices, and let $B,C_{J_0}>0$ be real
constants. Suppose that $J_0\subset\{1,\dots,d_n\}$ for all
$n\in\N$ and
\[C_{J_0}({\cal F}_\kappa(\Omega_{{\mathscr D}_n}))\geq C_{J_0}>0.\]
 
Then we have with the same definitions of $r_n$, $\phi_n$, $q_n$, and
$p_n$ as in Theorem~\ref{theoConvC1} that $q_n\to g'$ and $g_n\to g$ in
$C^\infty(D)$ as $n\to\infty$, that is
discrete partial derivatives of all orders of $q_n$ and $g_n$ converge
uniformly on compact subsets to their smooth counterparts.
\end{theorem}

Simple examples of sequences of quasicrystallic circle
patterns for this theorem are subgraphs of the suitably
scaled infinite regular square grid or
hexagonal circle patterns or subgraphs of 
suitably scaled infinite rhombic embeddings constructed in
Examples~\ref{exquasi} and~\ref{exquasihat} (see Figure~\ref{figExCirc}).
Simply connected parts of these
rhombic embeddings which are large enough satisfy the
conditions $C_{J}({\cal F}_\kappa(\Omega_{{\mathscr D}_n}))\geq C_0>0$ for all
subsets $J\subset\{1,\dots, d\}$, where the constant $C_0$
only depends on the construction parameters.
This is a consequence of the simple combinatorics or of the modified
construction in Example~\ref{exquasihat}.

\begin{remark}
Similarly as for Theorem~\ref{theoConvC1}, the proof
of Theorem~\ref{theoConvQCinfty}
shows that we have in fact a priori bounds in $O(\eps_n)$ on compact subset of
$D$ for the difference of directional derivatives of $\log g'$ (and
thus of $g'$ and $g$) and corresponding 
discrete partial derivatives of $\log\frac{r_n}{\eps_n}+i(\varphi_n-\phi_n)$
(and thus $q_n$ and $g_n$ respectively). The constant in the estimation depends
on the order of the partial derivatives and on the compact subset.

Furthermore, for the 'very regular' case 
considered in Remark~\ref{remeps2}
the a priori bounds on the above partial
derivatives have order $O(\eps_n^2)$ on compact subset of $D$
due to the improved estimation~\eqref{eqlem1eps2}.
\end{remark}

\begin{remark}\label{remCinfty}
There is an analogous version of Theorem~\ref{theoConvQCinfty}
of $C^\infty$-convergence for quasicrystallic circle patterns
with Neumann boundary conditions.
\end{remark}

For the proof of Theorem~\ref{theoConvQCinfty}, we first define the comparison
function $w_n$ for the circle pattern ${\mathscr C}_n$
according to~\eqref{eqdefw} and extend it to ${\cal F}_\kappa(\Omega_{{\mathscr
D}_n})$.
This extension is again
denoted by $w_n$. The restriction of $w_n$ to white vertices is called
{\em extended radius function} and denoted by $r_n$ as the original radius
function for ${\mathscr C}_n$.
We also extend $|g'|$ to ${\cal  F}_\kappa(\Omega_{{\mathscr D}_n})$
by defining $|g'({\mathbf z})|$ to be
the value $|g'(z)|$ at the projection $z$ of ${\mathbf z}\in\Z^d$ onto the plane
of the rhombic embedding ${\mathscr D}_n$. If $n$ is big enough, which will
be assumed in the following, then $z\in D$ or $z\in W\setminus D$ and the
distance of $z$ to $D$ is bounded independently of $n$.

Denote $t_n=\log(r_n/R_n) =\log(r_n/\eps_n)$ and $h_n=\log|g'|$ as in
Section~\ref{secDirichlet}. Using the
extensions of $r_n$ and $|g'|$, these functions are defined on all white
vertices of ${\cal F}_\kappa(\Omega_{{\mathscr D}_n})$.
Furthermore we have the following extension of Lemma~\ref{lem1}.

\begin{lemma}\label{corlem1}
The estimation
\[ h_n({\mathbf z})-t_n({\mathbf z})=\Od(\eps_n).\]
holds for all white vertices ${\mathbf z}\in V({\cal
F}_\kappa(\Omega_{{\mathscr D}_n}))$. The constant in the $\Od$-notation may
depend on $\kappa$ and on the dimension $d_n$ of ${\mathscr D}_n$.
\end{lemma}
\begin{proof}
 Let ${\mathbf z}\in V({\cal F}(\Omega_{{\mathscr D}_n}))$ be a white vertex.
By definition of ${\cal F}(\Omega_{{\mathscr D}_n})$ there is a combinatorial
surface $\Omega'({\mathbf z})$ containing ${\mathbf z}$ with the same boundary
curve as $\Omega_{{\mathscr D}_n}$. Lemma~\ref{lem2patt} implies that we can
define an embedded circle pattern ${\mathscr C}'$ using the values of $w_n$ on
$\Omega_{{\mathscr D}_n}$. Now the claim follows from Lemma~\ref{lem1}.

For ${\mathbf z}\in V({\cal F}_\kappa(\Omega_{{\mathscr D}_n}))\setminus
V({\cal F}(\Omega_{{\mathscr D}_n}))$ observe that equation~\eqref{eqw} may be
used to
extend the estimation of $h_n-t_n$. Each steps adds an error of order $\eps_n$,
therefore the final constant depends on $\kappa$ and on $d_n$.
\end{proof}

Our main aim is to estimate the
partial derivatives of the extended radius function in $\Z^{d_n}$.
Such partial derivatives can generally be considered in direction of the
vectors ${\mathbf v}= \pm {\mathbf e}_{j_1}\pm {\mathbf e}_{j_2}$ for $0\leq
j_1,j_2\leq d_n$
such that ${\mathbf e}_{j_1}$ and ${\mathbf e}_{j_2}$ are not collinear. Let
${\mathbf v}_{1},\dots, {\mathbf
v}_{2d_n(d_n-1)}$ be an enumeration of these vectors and set ${\mathbf V}_n=\{
{\mathbf v}_{1},\dots, {\mathbf v}_{2d_n(d_n-1)} \}$.
The corresponding enumeration of the directions $v= \pm
{a}_{j_1}^{(n)}\pm {a}_{j_2}^{(n)}$ in ${\mathscr D}_n$  for $0\leq j_1,j_2\leq
d$ is denoted by $v_1,\dots v_{2d_n(d_n-1)}$.

For any function $h$ on white vertices ${\mathbf z}\in\Z^d$ and/or in $z\in\C$
define {\em discrete partial derivatives} in direction ${\mathbf v}_i$ or
$v_i$ by
\begin{align*} 
  \partial_{{\mathbf v}_i}h({\mathbf z})
  &=\frac{h({\mathbf z}+{\mathbf v}_i)-h({\mathbf z})}{\eps_n|v_i|} 
  \qquad\text{and}\qquad \partial_{{v}_i}h({z})
  =\frac{h({z}+\eps_n{v}_i)-h({z})}{\eps_n|v_i|} 
\end{align*}
respectively.
Furthermore, we call a direction ${\mathbf v}_i$ or $v_i$ to be {\em contained
in} $\Omega_{{\mathscr D}_n}$ or ${\mathscr D}_n$ at a vertex $\hat{\mathbf
z}$ or $z$ respectively if there is a two-dimensional facet of
$\Omega_{{\mathscr D}_n}$ incident to $\hat{\mathbf z}$ whose diagonal incident
to $\hat{\mathbf z}$ is parallel to ${\mathbf v}_i$, that is
$\{\hat{\mathbf z}+\lambda {\mathbf v}_i:
\lambda\in[0,1]\}\subset \Omega_{{\mathscr D}_n}$.

Corresponding to these partial derivatives we use a scaled version
of the Laplacian in~\eqref{eqLap}. For a function $\eta$ and an
interior vertex $z$ with incident vertices $z_1,\dots, z_L$ define the {\em
discrete Laplacian} by
\begin{equation}\label{deflap2}
  \Delta^{\eps_n} \eta (z):=\Delta^{\eps_n}_{{\mathbf v}_{\mu_1}, \dots,{\mathbf
v}_{\mu_L}} \eta (z) :=\frac{1}{\eps_n^2}\sum_{j=1}^L 2
  f'_{\alpha_n([z,z_j])}(0)(\eta(z_j)-\eta(z)).
\end{equation}
Here $v_{\mu_j}=v([z,z_j])=(z_j-z)/\eps_n$ ($j=1,\dots,L$) and the notation
$\Delta^{\eps_n}_{{\mathbf v}_{\mu_1}, \dots,{\mathbf v}_{\mu_L}}$ emphasizes
the dependence of the Laplacian on
the directions $v([z,z_j])$ of the edges $[z,z_j]\in E(G_n)$.

Let $z_0\in V_{int}(G_n)$ be an interior vertex and let ${\mathbf v}_{\mu_1},
\dots,{\mathbf v}_{\mu_L}\in {\mathbf V}_n$ be the directions which correspond
to the directions of the edges of $G_n$ incident to $z_0$.
Let ${\mathbf z}_1\in {\cal F}_\kappa(\Omega_{{\mathscr D}_n})$ be a vertex such
that ${\mathbf z}_1 +{\mathbf v}_{\mu_i}\in {\cal F}_\kappa(\Omega_{{\mathscr
D}_n})$ for all $i=1,\dots, L$. 
Our next aim is to study $\Delta^{\eps_n}_{{\mathbf v}_{\mu_1}, \dots,{\mathbf
v}_{\mu_L}} t_n ({\mathbf z}_1)$. For this purpose we assume that we have
translated to ${\mathbf z}_1$ the facets of $\Omega_{{\mathscr D}_n}$ incident
to ${\mathbf z}_0$, that is we consider the (very small) monotone surface
consisting of the two-dimensional facets incident to ${\mathbf z}_1$ which
contain a diagonal $\{{\mathbf z}_1+\lambda {\mathbf v}_{\mu_i}:
\lambda\in[0,1]\}$ for $i=1,\dots, L$. Using the extension of the comparison
function $w_n$, the closed chain of these two-dimensional facets incident to
${\mathbf z}_1$ is mapped to a closed chain of kites. Thus we have
\[\sum_{l=1}^L 2f_{\alpha_{\mu_l}}(t_n({\mathbf z}_1+
  {\mathbf v}_{\mu_l})-t_n({\mathbf z}_1)) \in 2\pi\N ,\]
where $\alpha_{\mu_l}$ denotes the labelling of the two-dimensional facet
containing ${\mathbf z}_1$ and ${\mathbf z}_1 +{\mathbf v}_{\mu_l}$.
By assumption $\sum_{l=1}^L 2f_{\alpha_{\mu_l}}(0)=2\pi$ and
we know that $t_n({\mathbf z}_1+
  {\mathbf v}_{\mu_l})-t_n({\mathbf z}_1) =\Od(\eps_n)$ by Lemma~\ref{corlem1}.
Since the intersection angles and thus the maximum number of
neighbors are uniformly bounded, we deduce that
\[\Biggl(\sum_{l=1}^L 2f_{\alpha_{\mu_l}}(t_n({\mathbf z}_1+
  {\mathbf v}_{\mu_l})-t_n({\mathbf z}_1)) \Biggr) -2\pi =0\]
if $\eps_n$ is small enough.
Using a Taylor expansion about $0$, we obtain
\begin{align}
  \Delta^{\eps_n}_{{\mathbf v}_{\mu_1}, \dots,{\mathbf
v}_{\mu_L}} t_n({\mathbf z}_1) &=
   \frac{2}{\eps_n^2}\left(-\sum_{l=1}^L \sum_{m=3}^\infty 
    \frac{f^{(m)}_{\alpha_{\mu_l}}(0)}{m!} (t_n({\mathbf z}_1+
  {\mathbf v}_{\mu_l})-t_n(z))^m\right) \nonumber\\ 
  &= \eps_n \left(-2\sum_{l=1}^L \sum_{m=3}^\infty
    \frac{f^{(m)}_{\alpha_{\mu_l}}(0)}{m!} 
  |v_{\mu_l}|^m(\partial_{{\mathbf v}_{\mu_l}} t_n({\mathbf
z}_1))^m\eps_n^{m-3}\right)\nonumber\\
  &=: \eps_n F_{{\mathbf v}_{\mu_1}, \dots,{\mathbf
v}_{\mu_L}} (\eps_n, \partial_{{\mathbf v}_{\mu_1}}t_n, \dots,
  \partial_{{\mathbf v}_{\mu_L}}t_n; {\mathbf z}_1), \label{eqdefFv} 
\end{align}
Note that $F_{{\mathbf v}_{\mu_1}, \dots,{\mathbf v}_{\mu_L}}$ is a
$C^\infty$-function in the variables
$\eps_n,\partial_{{\mathbf v}_{\mu_1}}t_n, \dots, \partial_{{\mathbf
v}_{\mu_L}}t_n$. This fact will be important for the proof of
Lemma~\ref{lemInd} below.

For further use we introduce the following notation.
Let $K$ be a compact subset of $D$ and let $M>0$.
Denote by $\Omega_{{\mathscr D}_n}^{K,M}$ the part of
$\Omega_{{\mathscr D}_n}$ with vertices of combinatorial distance
bigger than $M$ to the boundary and 
whose corresponding vertices $z\in V({\mathscr D}_n)$ lie in $K$.
Let $J\subset\{1,\dots,d\}$
contain at least two different indices. In order to consider
partial derivatives within $K$ in the directions $\pm {\mathbf e}_{j_1} \pm
{\mathbf e}_{j_2}$, where $j_1,j_2\in J$ and $j_1\not=j_2$, we attach a
ball $U_{J}(\hat{\mathbf z},M)$ at each of these points:
\begin{equation}
U_J(K,M,\Omega_{{\mathscr D}_n})= \bigcup_{\hat{\mathbf z}\in
V(\Omega_{{\mathscr D}_n}^{K,M})} U_J(\hat{\mathbf
z},M).
\end{equation}
Note that if $M\leq B_J(\hat{\mathbf z})$, for example if $M\leq
C_J({\cal F}_\kappa(\Omega_{{\mathscr D}_n})) d(\hat{\mathbf
z},\partial \Omega_{{\mathscr D}_n})-1$ for all $\hat{\mathbf z}\in
V(\Omega_{{\mathscr D}_n}^{K,M})$, then
$U_J(K,M,\Omega_{{\mathscr D}_n}) \subset {\cal
F}_\kappa(\Omega_{{\mathscr D}_n})$.

Furthermore we define $K+d_0$ to be the compact $d_0$-neighborhood of
the compact set $K\subset \C$, that is 
\begin{equation*}
 K+d_0 =\{z\in\C : \mathbb{e}(K,z)\leq d_0\},
\end{equation*}
where $\mathbb{e}$ denotes the Euclidean distance between a point and a compact
set or between two compact subsets of $\C$.

The following lemma is important for our argumentation, as it gives an
estimation of a partial derivative using estimations of the function and its
Laplacian.

\begin{lemma}[Regularity Lemma]\label{lemRegIso1}
Let ${\mathscr D}$ be a quasicrystallic rhombic embedding with associated
graph $G$ and labelling $\alpha$.
 Let $W\subset V(G)$ and let $u:W\to \R$ be any function. Let
$M(u)=\max_{v\in W_{int}}|\Delta u(v)/(4F^*(v))|$, where
\begin{equation*}
F^*(v)=\frac{1}{4} \sum_{[z,v]\in E(G)} c([z,v])|z-v|^2
= \frac{1}{2}\sum_{[z,v]\in E(G)} \sin\alpha([z,v])
\end{equation*}
is the area of the face of the dual graph $G^*$ corresponding to the vertex
$v\in V_{int}(G)$.
There are constants $C_5,C_6>0$, independent of $W$ and $u$, such that
\begin{equation}
 |u(x_0)-u(x_1)|\rho \leq C_5\|u\|_W +\rho^2 C_6 M(u)
\end{equation}
for all vertices $x_1\in W$ incident to $x_0$, where $\rho$ is the Euclidean
distance of $x_0$ to the boundary $W_\partial$.
\end{lemma}
A proof is given in the appendix, see Lemma~\ref{lemRegIso}. As a direct
application we get

\begin{lemma}\label{lemInd}
Let $K\subset D$ be a compact set, $n_0\in\N$ and 
$0<d_0<\mathbb{e}(K,\partial {\mathscr D}_n)$ for all $n\geq n_0$.

Let $k\in\N_0$ and let ${\mathbf v}_{i_1},\dots,{\mathbf v}_{i_k}\in
\bigcap_{n\geq n_0}{\mathbf V}_n$ be $k$ not
necessarily different directions. Let $J\subset\{1,\dots,d\}$ be a minimal
subset of indices such that
$\{{\mathbf v}_{i_1},\dots,{\mathbf v}_{i_k}\} \subset \text{span}\{{\mathbf
e}_j :j\in J\}$.
Let $B_0,C_0\geq 0$ be some constants.
Assume that all discrete partial derivatives using at most $k$ of the directions
${\mathbf v}_{i_1},\dots,{\mathbf v}_{i_k}$ exist on
$U_0=U_J(K+d_0,B_0,\Omega_{{\mathscr D}_n})$ and are bounded on $U_0$ by
$C_0\eps_n$ for all $n\geq n_0$.

Let $n\geq n_0$ and
let $\Omega_{\hat{\mathscr D}_n}$ be a two-dimensional monotone combinatorial
surface. Let $\hat{\mathscr D}_n$ be the corresponding rhombic embedding with
edge lengths $\eps_n$ and such that
\[\{\hat{\mathbf z}\in V(\Omega_{\hat{\mathscr D}_n})
: z\in V(\hat{\mathscr D}_n)\cap (K+d_0)\} \subset U_0.\]
Let $z_0\in V_w(\hat{\mathscr D}_n) \cap (K+d_0/2)$ be a white vertex such that
$\mathbb{e}(z,\partial \hat{\mathscr D}_n)\geq d_0/2$. 
Let ${\mathbf v}_{i_{k+1}}\in {\mathbf V}$ be a direction contained in
$\Omega_{\hat{\mathscr D}_n}$ at $\hat{\mathbf z}_0$.
Then there is a constant $C_1$, depending on $K$, $D$, $g$, and on the constants
$d_0$, $B_0$, $C_0$, $\kappa$, $C$, but not on $z_0$, ${\mathbf
v}_{i_{k+1}}$, and $n$, such that
\begin{equation*}
 |\partial_{{\mathbf v}_{i_{k+1}}} \partial_{{\mathbf
v}_{i_{k}}}\cdots \partial_{{\mathbf v}_{i_1}}(t_n-h_n)(z_0)|\leq C_1\eps_n.
\end{equation*}
\end{lemma}
\begin{proof}
The proof is an application of the Regularity Lemma~\ref{lemRegIso1} for the
function
$u=\partial_{{\mathbf v}_{i_{k}}}\cdots \partial_{{\mathbf v}_{i_1}}(t_n-h_n)$
and the part of the rhombic embedding $\hat{\mathscr D}_n$ contained in
$K+d_0$. Note that the edge lengths of $\hat{\mathscr D}_n$ are $\eps_n$ and we
suppose that $\mathbb{e}(z,\partial \hat{\mathscr D}_n)\geq d_0/2$.

By assumption we have $\|u\|_{V(\hat{\mathscr D}_n)\cap (K+d_0)}\leq C_0\eps_n$.
As $h_n$ is a $C^\infty$-function,
this implies that $\partial_{{\mathbf v}_{i}} \partial_{{\mathbf
v}_{i_{k}}}\cdots \partial_{{\mathbf v}_{i_1}}t_n = \Od(1)$ on
$U_0$ for all possible directions ${\mathbf v}_i\in{\mathbf V}$  whenever this
partial
derivative is defined in ${\cal F}_\kappa(\Omega_{\hat{\mathscr D}_n})$.

Let $z_1\in V_{int}(\hat{G}_n)\cap (K+d_0)$ be an interior white vertex of
$\hat{\mathscr D}_n$ and let ${\mathbf v}_{\mu_1},
\dots,{\mathbf v}_{\mu_L}\in{\mathbf V}_n$ be the directions corresponding to
the directions
of the edges of $\hat{G}_n$ incident to $z_1$. 
Note that
\[\Delta^{\eps_n}_{{\mathbf v}_{\mu_1}, \dots,{\mathbf v}_{\mu_L}} \!u =
  \Delta^{\eps_n}_{{\mathbf v}_{\mu_1}, \dots,{\mathbf v}_{\mu_L}}
\partial_{{\mathbf v}_{i_{k}}}\!\!\!\cdots \partial_{{\mathbf v}_{i_1}}(t_n-h_n)
  =\partial_{{\mathbf v}_{i_{k}}}\!\!\!\cdots
\partial_{{\mathbf v}_{i_1}} \Delta^{\eps_n}_{{\mathbf v}_{\mu_1},
\dots,{\mathbf v}_{\mu_L}}(t_n-h_n). \]
From the above consideration in~\eqref{eqdefFv} we use
   that $\Delta^{\eps_n}_{{\mathbf v}_{\mu_1}, \dots,{\mathbf v}_{\mu_L}}
 t_n=\eps_n F_{{\mathbf v}_{\mu_1}, \dots,{\mathbf v}_{\mu_L}}$, where
$F_{{\mathbf v}_{\mu_1}, \dots,{\mathbf v}_{\mu_L}}$ is a
  $C^\infty$-function in the variables $\eps_n,
  \partial_{{\mathbf v}_{\mu_1}} t_n, \dots,\partial_{{\mathbf v}_{\mu_L}}
t_n$.
From our assumptions we know that all partial 
  derivatives $\partial_{{\mathbf v}_{\mu_i}} \partial_{{\mathbf
v}_{i_{k}}}\cdots \partial_{{\mathbf v}_{i_1}}t_n(z_1)$ and those containing
less than $k+1$ of the derivatives $\partial_{{\mathbf v}_{\mu_i}},
\partial_{{\mathbf v}_{i_{k}}},\dots, \partial_{{\mathbf v}_{i_1}}$ are defined
and uniformly bounded by a constant independent of $z_1$ and $\eps_n$. Thus
\[\left|\partial_{{\mathbf v}_{i_{k}}}\cdots
\partial_{{\mathbf v}_{i_1}} \Delta^{\eps_n}_{{\mathbf v}_{\mu_1},
\dots,{\mathbf v}_{\mu_L}}t_n(z_1)\right| \leq C_{{\mathbf v}_{\mu_1},
\dots,{\mathbf v}_{\mu_L}} \eps_n,\]
where the constant $C_{{\mathbf v}_{\mu_1}, \dots, {\mathbf v}_{\mu_L}}$ does
not depend on $z_1$ and $\eps_n$. As $h_n$ is a harmonic $C^\infty$-function,
we obtain by similar reasonings as in the proof of Lemma~\ref{lem1} that
\[\partial_{{\mathbf v}_{i_{k}}}\cdots
\partial_{{\mathbf v}_{i_1}} \Delta^{\eps_n}_{{\mathbf v}_{\mu_1},
\dots,{\mathbf v}_{\mu_L}}h_n =\Od(\eps_n). \]
As ${\mathbf V}_n$ is a finite set, we deduce
that $\Delta^{\eps_n} u$ is uniformly bounded on $V_{int}(\hat{G}_n)\cap
(K+d_0)$
by $C_2\eps_n$ for some constant $C_2$ independent of $\eps_n$.
Now the Regularity Lemma~\ref{lemRegIso1} gives the claim.
\end{proof}

The following corollary of the preceeding lemma constitutes the crucial step in
our proof of Theorem~\ref{theoConvQCinfty}.

\begin{lemma}\label{lemEstpart}
 Let $K\subset D$ be a compact set and let $0<d_1<\mathbb{e}(K,\partial
{D})$.
Let $n_0\in\N$ be such that $K+d_1$ is covered by the rhombi of ${\mathscr
D}_n$ for all $n\geq n_0$.
Then there is a constant $C_1 =C_1(K,C_{J_0}({\cal F}_\kappa(\Omega_{{\mathscr
D}_n})))>0$ such that for all $z\in K+d_1$
\[C_1\eps_n^{-1}\leq C_{J_0}({\cal F}_\kappa(\Omega_{{\mathscr
D}_n})) \cdot d(z,\partial {\mathscr D}_n) -1.\]

Furthermore, let $k\in\N_0$ and let ${\mathbf v}_{i_1},\dots,{\mathbf
v}_{i_k}\in {\mathbf V}$
be $k$ (not necessarily different) directions such that ${\mathbf v}_{i_l}\in
\text{span}\{e_j : j\in J_0\}$ for all $l=1,\dots,k$.
Then there are constants $n_0\leq n_1(k,K) \in\N$ and $C(k,K)>0$ which may
depend on $k$, $K$, $d_1$, $D$, $g$, $\kappa$, $C_{J_0}({\cal
F}_\kappa(\Omega_{{\mathscr D}_n}))$, but not on $\eps_n$, such that for all
$n\geq n_1(k,K)$ we have
\begin{equation}\label{eqestpart}
 \|\partial_{{\mathbf v}_{i_{k}}}\cdots \partial_{{\mathbf
v}_{i_1}}(t_n-h_n)\|_{U_k}\leq C(k,K)\eps_n,
\end{equation}
where $U_k= U_{J_0}(K+2^{-k}d_1, 2^{-k}C_1\eps_n^{-1}, \Omega_{{\mathscr
D}_n})$.
\end{lemma}
\begin{proof}
The existence of the constant $C_1$ follows from the fact that $d(z,\partial
{\mathscr D}_n)/\eps_n$ is bounded from below for $z\in (K+d_1)$ since the
distance $\mathbb{e}(K+d_1,\partial {\mathscr D}_n)>0$ is positive and the
angles of the rhombi are uniformly bounded.

The proof of estimation~\eqref{eqestpart} uses induction on the number of
partial derivatives $k$.
For $k=0$ the claim has been shown in Lemma~\ref{corlem1}.

Let $k\in\N_0$ and assume that the claim is true for all $\nu\leq k$. 
Let ${\mathbf v}_{i_{k+1}}\in {\mathbf V}$ be a direction with ${\mathbf
v}_{i_{k+1}}\in \{\pm
{\mathbf e}_{j_1}\pm {\mathbf e}_{j_2} \}\subset \text{span}\{{\mathbf e}_j :
j\in J_0\}$.
Using the induction hypotheses, we can apply Lemma~\ref{lemInd} for $U_0=U_k$,
$d_0 =2^{-k}d_1$, $\Omega_{\hat{{\mathscr D}_n}}
=U_{\{j_1,j_2\}}(K+2^{-k}d_1, 2^{-k}C_1\eps_n^{-1}, \Omega_{{\mathscr D}_n})
\subset U_k$ and the corresponding rhombic embedding $\hat{{\mathscr D}_n}$
obtained by projection, and $z_0\in \hat{{\mathscr D}_n} \cap (K+2^{-k-1}d_1)$.
This completes the induction step and the proof.
\end{proof}

\begin{proof}[Proof of Theorem~\ref{theoConvQCinfty}]
Identify $\C$ with $\R^2$ in the
standard way and fix two orthogonal unit vectors
$e_1,e_2$. Define discrete partial derivatives $\partial_{e_1},  
\partial_{e_2}$ in these directions using the discrete partial derivatives
in two orthogonal directions 
$v_{i_1}=a_{j_1}^{(n)}+a_{j_2}^{(n)}$ and $v_{i_2}=a_{j_1}^{(n)}-a_{j_2}^{(n)}$
for
$j_1,j_2\in J_0$. This definition depends on the choice
of $a_{j_1}^{(n)},a_{j_2}^{(n)}$, which may be different for each $n$, but this
does not affect the proof.
As the possible intersection angles are bounded and as $h_n$ is a
$C^\infty$-function, we deduce that
\[\|\partial_{e_{j_k}}\cdots \partial_{e_{j_1}}h_n
-\partial_{{j_k}}\cdots \partial_{{j_1}}h_n\|_K \leq C_1(k,K)\eps_n\]
on every compact set $K$ for $j_k,\dots,j_1\in\{1,2\}$.
Here $\partial_1,\partial_2$ denote the standard
partial derivatives associated to $e_1,e_2$ for smooth functions and $C_1(k,K)$
is a constant which depends only on $K$, $k$, and $g$. Lemma~\ref{lemEstpart}
implies that 
\begin{equation*}
\|\partial_{e_{j_k}}\cdots
\partial_{e_{j_1}}t_n- \partial_{e_{j_k}}\cdots
\partial_{e_{j_1}}h_n\|_{U_{J_0}(K+2^{-k}d_1, 2^{-k}C_1\eps_n^{-1},
\Omega_{{\mathscr D}_n})} \leq C_2(k,K)\eps_n
\end{equation*}
if $n$ is big enough. Using a version of Lemma~\ref{lem2} with error of order
$\Od(\eps_n)$, we deduce that
$t_n+i(\varphi_n-\phi_n)$ converges to $\log g'$ in $C^\infty(D)$. Now the
convergence of $q_n$ and $g_n$ follows by similar arguments as in the proof of
Theorem~\ref{theoConvC1}.
\end{proof}

\appendix

\section[Appendix: Regularity of solutions of
discrete elliptic equations]{Appendix: Properties of discrete Green's function
and regularity of solutions of
discrete elliptic equations}\label{secPropGreen}

In order to prove the Regularity Lemma~\ref{lemRegIso1} (see Lemma
~\ref{lemRegIso}), we present some results in discrete potential theory on
quasicrystallic rhombic embeddings which are derived from a suitable asymptotic
expansion of a discrete Green's function.

Throughout this appendix, we assume that $\mathscr D$ is a (possibly
infinite) simply connected quasicrystallic rhombic embedding of a b-quad-graph
with edge directions ${\cal A} =\{\pm a_1,\dots,\pm a_d\}$.
Also, the edge lengths of  $\mathscr D$ are supposed
to be normalized to one.
Let $G$ be the associated graph built from white vertices.

Fix some interior vertex $x_0\in V_{int}(G)$. 
Following Kenyon~\cite{Ke02} and Bobenko, Mercat, and Suris~\cite{BMS05},
we define the
{\em discrete Green's function} ${\cal G}(x_0,\cdot):V(G)\to\R$ by
\begin{equation}\label{defGreen}
{\cal G}(x_0,x)= -\frac{1}{4\pi^2 i} \int_{\Gamma} \frac{\log
(\lambda)}{2\lambda} e(x;\lambda) d\lambda
\end{equation}
for all $x\in V(G)$. Here $
e(x;z)=\prod_{k=1}^d\left(\frac{z+a_k}{z-a_k}\right)^{n_k}$ is the {\em discrete
exponential function},
where ${\mathbf n}=(n_1,\dots,n_d)=\hat{\mathbf x}- \hat{\mathbf x}_0\in \Z^d$
and $\hat{\mathbf x}, \hat{\mathbf x}_0\in V(\Omega_{\mathscr D})$ correspond
to $x,x_0\in V({\mathscr D})$ respectively.
The integration path $\Gamma$ is a collection of $2d$ small loops, each one
running counterclockwise around one of the points $\pm a_k$ for $k=1,\dots, d$.
The branch of
$\log(\lambda)$ depends on $x$ and is chosen as follows.
Without loss of generality, we assume that the circular order of the points of
$\cal A$ on the positively oriented unit circle $\Sp^1$ is $a_1,\dots,
a_d,-a_1,\dots, -a_d$. Set $a_{k+d}=-a_k$ for $k=1,\dots,d$ and define
$a_m$ for all $m\in\Z$ by $2d$-periodicity. To each
$a_m=\text{e}^{i\theta_m}\in\Sp^1$ we assign a certain value of the argument
$\theta_m\in\R$: choose $\theta_1$ arbitrarily and then use the rule
\[\theta_{m+1}-\theta_m\in(0,\pi)\qquad \text{for all }m\in\Z. \]
Clearly we then have $\theta_{m+d}=\theta_m+\pi$.
The points $a_m$ supplied with the arguments $\theta_m$ can be considered as
belonging to the Riemann surface of the logarithmic function (i.e.\ a branched
covering of the complex $\lambda$-plane). 
Since $\Omega_{\mathscr D}$ is a monotone surface, there is an $m\in\{1,\dots,
2d\}$ and a directed path from $x_0$ to $x$ in $\mathscr D$ such that the
directed edges of this path are contained in $\{a_m,\dots,a_{m+d-1}\}$,
see~\cite[Lemma~18]{BMS05}. Now, the branch of
$\log(\lambda)$ in~\eqref{defGreen} is chosen such that
\[ \log(a_l)\in[i\theta_m,i\theta_{m+d-1}],\qquad l=m,\dots,m+d-1.\]

Remember Definition~\ref{defLap} of the discrete Laplacian and the
representation of its weights $c([z_1,z_2])= 2f_{\alpha([z_1,z_2])}'(0)$
in~\eqref{eqLapc}. Then there holds

\begin{lemma}[{\cite[Theorems~7.1 and~7.3]{Ke02}}]
The discrete Green's function ${\cal G}(x_0,\cdot)$ defined in
equation~\eqref{defGreen} has the following properties.
\begin{enumerate}[(i)]
 \item $\Delta {\cal G}(x_0,v)= -\delta_{x_0}(v)$, where the Laplacian is taken
with respect to the second variable.
\item ${\cal G}(x_0,x_0)=0$.
\item ${\cal G}(x_0,v)=\Od(\log(|v-x_0|))$.
\end{enumerate}
\end{lemma}
Note that  ${\cal G}(x_0,\cdot)$ may also be defined by these three conditions.

\subsection{Asymptotics for discrete Green's function}
Kenyon derived in~\cite{Ke02} an
asymptotic development for the discrete Green's function using
standard methods of complex analysis. His result can be
slightly strengthened to an error of order
$\Od(1/|v-x_0|^2)$. Note, that there is the summand $-\log2/(2\pi)$ missing in
Kenyon's formula (but not in his proof).

\begin{theorem}[cf.\ {\cite[Theorem~7.3]{Ke02}}]\label{theoGreen}
For $v\in V(G)$ there holds
\begin{equation}\label{eqestGreen}
{\cal G}(x_0,v)= -\frac{1}{2\pi}\log
(2|v-x_0|) -\frac{\gamma_{\text{Euler}}}{2\pi}
+\Od\left(\frac{1}{|v-x_0|^2}\right).
\end{equation}
Here $\gamma_{\text{Euler}}$ denotes the {\em Euler $\gamma$ constant}.
\end{theorem}
\begin{proof}
Consider
a directed path $x_0=w_0,\dots,w_k=v$ in $\mathscr D$ from $x_0$ to $v$ 
such that the directed edges 
of this path are contained in $\{a_m,\dots,a_{m+d-1}\}$ for some
$1\leq m\leq 2d$ as above. 
Note that $k$ is even since $x_0$ and $v$ are both white vertices of $\mathscr
D$. The integration path~$\Gamma$ in~\eqref{defGreen} can be
deformed into a connected contour lying on
a single leaf of the Riemann surface of the logarithm, in particular 
to a simple closed curve $\Gamma_1$
which surrounds the set $\{a_m,\dots,a_{m+d-1}\}$ in a counterclockwise sense
and has the origin and a ray ${\mathscr R}=\{s\text{e}^{i\tilde{\theta}}:
s>0\}$ in its exterior. We also assume that $\Gamma_1$ is
contained in the sector $\{ z=r\text{e}^{i\varphi} : r> 0,\ \varphi\in
[\tilde{\theta}+\frac{\pi}{2} +\eta, \tilde{\theta}+\frac{3\pi}{2} -\eta]\}$
for some $\eta>0$ independent of $v$, $x_0$, and $m$.
This is possible due to the fact that
$\theta_{m+d-1}-\theta_m<\pi-\delta$ for some $\delta>0$ independent of $v$, 
$x_0$, and $m$. 

Let $N=|v-x_0|$. Take $0<\varrho_1\ll 1/N^3$ and $\varrho_2\gg N^3$, but not
exponentially smaller than $1/N$ or bigger than $N$ respectively.
The curve $\Gamma_1$ is again homotopic to a curve $\Gamma_2$ which runs
counterclockwise around the circle of radius $\varrho_2$ about the origin from
the angle $\tilde{\theta}$ to $\tilde{\theta}+2\pi$, then along the ray
${\mathscr
R}$ from $\varrho_2\text{e}^{i\tilde{\theta}}$ to
$\varrho_1\text{e}^{i\tilde{\theta}}$, then clockwise around the circle of
radius $\varrho_1$ about the origin from the angle $\tilde{\theta}+2\pi$ to
$\tilde{\theta}$, and finally back along the ray ${\mathscr R}$ from
$\varrho_1\text{e}^{i\tilde{\theta}}$ to
$\varrho_2\text{e}^{i\tilde{\theta}}$.
Without loss of generality, we assume that  ${\mathscr R}$ is the negative real
axis.

Kenyon showed in~\cite{Ke02} that the integrals along the circles of radius
$\varrho_1$ and $\varrho_2$ give
\[(-1)^k\frac{\log\varrho_1}{4\pi}(1+\Od(N\varrho_1))
-\frac{\log\varrho_2}{4\pi}(1+\Od(N/\varrho_2)).\]

The difference between the value of $\log z$ above and below the negative real
axis is
$2\pi i$. Thus the integrals along the negative real axis can
be combined into 
\[-\frac{1}{4\pi} \int_{-\varrho_2}^{-\varrho_1} \frac{1}{z} \prod_{j=0}^{k-1}
\frac{z+b_j}{z-b_j} dz,\]
where $b_j=w_{j+1}-w_j\in {\cal A}$ is the directed edge
from $w_j$
to $w_{j+1}$ and $k=\Od(N)$ is the number of edges of the path. This
integral can be split into three parts: from  $-\varrho_2$ to $-\sqrt{N}$, from
$-\sqrt{N}$ to $-1/\sqrt{N}$, and from $-1/\sqrt{N}$ to $-\varrho_1$.

The integral is neglectible for the intermediate range
because
\[\left| \frac{t+\text{e}^{i\beta}}{t-\text{e}^{i\beta}}\right| \leq
\text{e}^{2t\cos\beta /(t-1)^2} \]
for negative $t<0$ and due to our assumptions.

For small $|t|$ we have
\[\prod_{j=0}^{k-1}\frac{t+b_j}{t-b_j}
=\prod_{j=0}^{k-1}\frac{t\bar{b}_j+1}{t\bar{b}_j-1} =(-1)^k
\text{e}^{2\sum_{j=0}^{k-1}\bar{b}_jt}(1+\Od(kt^3)),\]
using the Neumann series and a Taylor expansion.
Thus the integral near the origin is
\[ -\frac{(-1)^k}{4\pi}\left( \int_{-1/\sqrt{N}}^{-\varrho_1}
\frac{\text{e}^{2(\bar{v}-\bar{x}_0)t}}{t} dt + 
\int_{-1/\sqrt{N}}^{-\varrho_1}\Od(kt^3)
\frac{\text{e}^{2(\bar{v}-\bar{x}_0)t}}{t} dt \right).\]
Applying similar reasonings and estimations as in Kenyon's
proof, we obtain
\[-\frac{(-1)^k}{4\pi}\left(\log(2\varrho_1 (\bar{v}-\bar{x}_0))
+\gamma_{\text{Euler}} \right) + \Od\left(\frac{1}{N^2}\right).\]
Here $\gamma_{\text{Euler}}$ denotes the Euler $\gamma$ constant.

For large $|t|$ the estimations are very similar. Since
\[\prod_{j=0}^{k-1}\frac{t+b_j}{t-b_j}
=\prod_{j=0}^{k-1}\frac{1+b_jt^{-1}}{1-b_jt^{-1}}
=\text{e}^{2\sum_{j=0}^{k-1}b_jt^{-1}}(1+\Od(kt^{-3})),\]
we get
\begin{multline*}
-\frac{(-1)^k}{4\pi}\left( \int_{-\varrho_2}^{-\sqrt{N}}
\frac{\text{e}^{2(v-x_0)t^{-1}}}{t} dt + 
\int_{-\varrho_2}^{-\sqrt{N}} \Od(kt^{-3})
\frac{\text{e}^{2(v-x_0)t^{-1}}}{t} dt \right) \\
 = -\frac{1}{4\pi}\left(-\log\left(\frac{\varrho_2}{2(v-x_0)}\right)
+\gamma_{\text{Euler}} \right) + \Od\left(\frac{1}{N^2}\right).
\end{multline*}

Since $k$ is even, the sum of all the above integral parts is therefore
given by the right hand side of~\eqref{eqestGreen}.
\end{proof}

For a bounded domain we also define a discrete Green's function with
vanishing boundary values.
Let $W\subset V_{int}(G)$ be a finite subset of vertices. Denote by
$W_\partial\subset W$ the set of boundary vertices which are incident to at
least one vertex in $V(G)\setminus W$. Set $W_{int}=W\setminus W_\partial$ the
interior vertices of $W$. Let $x_0\in W_{int}$ be
an interior vertex. The {\em discrete Green's
function ${\mathcal G}_W(x_0,\cdot)$} is
uniquely defined by the following properties.
\begin{enumerate}[(i)]
 \item $\Delta {\cal G}_W(x_0,v)= -\delta_{x_0}(v)$ for all $v\in
W_{int}$, where the Laplacian is taken with respect to the second
variable.
\item ${\cal G}_W(x_0,v)=0$ for all $v\in W_\partial$.
\end{enumerate}

In the following, we choose $W$ to be a special disk-like set.
Let $x_0\in V(G)$ be a vertex and
let $\rho>2$. Denote the closed disk with center $x_0$ and radius $\rho$ by
$B_\rho(x_0)\subset \C$. Suppose that this disk is entirely covered by the
rhombic
embedding $\mathscr D$. Denote by $V(x_0,\rho)\subset V(G)$ the set of white
vertices lying within $B_\rho(x_0)$.
For $x_1\in V_{int}(x_0,\rho)$ we denote
\[{\cal G}_{x_0,\rho}(x_1,\cdot)= {\cal G}_{V(x_0,\rho)}(x_1,\cdot).\]

The asymptotics of the discrete Green's function $\cal G$ from
Theorem~\ref{theoGreen} can be used to derive the following estimations for
${\cal G}_{x_0,\rho}$.

\begin{proposition}\label{propGrhoRand}
There is a constant $C_1$, independent of $\rho$ and $x_0$, such that
\begin{equation*}
 |{\cal G}_{x_0,\rho}(x_0,v)|\leq C_1/\rho 
\end{equation*}
for all vertices $v\in V_{int}(x_0,\rho)$ which are incident to a boundary
vertex.

Furthermore,
there is a constant $C_2$, independent of $\rho$ and $x_0$, such that
for all interior vertices $x_1\in V_{int}(x_0,\rho)$ incident to $x_0$ and
all $v\in V(x_0,\rho)$ there holds
\begin{equation*}
 |{\cal G}_{x_0,\rho}(x_0,v)-{\cal G}_{x_0,\rho}(x_1,v)|\leq C_2/(|v-x_0|+1).
\end{equation*}
\end{proposition}
\begin{proof}
Consider the function $h_\rho(x_0,\cdot):V(x_0,\rho)\to \R$ defined by
\[h_\rho(x_0,v) ={\cal G}_{x_0,\rho}(x_0,v) -{\cal G}(x_0,v)
-{\textstyle \frac{1}{2\pi}}(\log(2\rho) +\gamma_{\text{Euler}}). \]
Then $h_\rho(x_0,\cdot)$ is harmonic on $V_{int}(x_0,\rho)$.
For boundary vertices
$v\in V_\partial (x_0,\rho)$ Theorem~\ref{theoGreen} implies that
$h_\rho(x_0,v)=\Od(1/\rho)$.
The Maximum Principle~\ref{MaxPrinzip} yields $|h_\rho(x_0,v)|\leq C/\rho$
for all $v\in V(x_0,\rho)$ and some constant $C$ independent of $\rho$ and $v$.
This shows the first estimation.

To prove the second claim, we also consider the harmonic function 
\[h_\rho(x_1,v)
={\cal G}_{x_0,\rho}(x_1,v) -{\cal G}(x_1,v)
-\frac{1}{2\pi}(\log(2\rho) +\gamma_{\text{Euler}})\]
for a fixed interior
vertex $x_1\in V_{int}(x_0,\rho)$ incident to $x_0$. By similar reasonings as
for $h_\rho(x_0,\cdot)$, we deduce
that  $|h_\rho(x_1,v)|\leq \tilde{C}/\rho$
for all $v\in V(x_0,\rho)$ and some constant $\tilde{C}$ independent of $\rho$
and $v$. Theorem~\ref{theoGreen} implies the desired estimation.
\end{proof}

\subsection{Regularity of discrete solutions of elliptic
equations}\label{secRegEllip}

In the following, we generalize and adapt some results of discrete potential
theory for discrete harmonic functions obtained by Duffin in~\cite[see
in particular Lemma~1 and Theorem~3--5]{Du53}. The proofs are very similar or
use ideas of the corresponding proofs in~\cite{Du53}.
Our aim is to obtain the Regularity Lemma~\ref{lemRegIso}.

\begin{lemma}[Green's Identity]\label{lemGreenId}
Let $W\subset V(G)$ be a finite subset of vertices.
Let $u,v:W\to\R$ be two functions. Then
\begin{equation}
 \sum_{x\in W_{int}} (v(x)\Delta u(x) - u(x)\Delta v(x)) =
\sum_{[p,q]\in E_\partial(W)}
c([p,q])(v(p)u(q) -u(p)v(q)),
\end{equation}
where $E_\partial(W) =\{[p,q]\in E(W) : p\in W_{int}, q\in W_\partial\}$.
\end{lemma}

\begin{corollary}[Representation of harmonic functions]\label{corEstharm}
Let  $u$ 
be a real valued harmonic function defined on $V(x_0,\rho)$. Then
\begin{equation*}
 u(x_0)=\sum_{q\in V_\partial (x_0,\rho)} c(q) u(q),
\end{equation*}
where
\begin{equation}\label{eqestbound}
c(q)=\sum_{\substack{p\in V_{int} (x_0,\rho) \\ \text{and } [p,q]\in E(G)}}
c([p,q]) {\cal G}_{x_0,\rho} (x_0,p) = \Od(1/\rho).
\end{equation}
\end{corollary}
The estimation in~\eqref{eqestbound} is a consequence of
Proposition~\ref{propGrhoRand} and of the boundedness of the weights
$c(e)$.

\begin{theorem}\label{theoGaussharm}
 Let  $u:V(x_0,\rho)\to \R$ be a non-negative harmonic function. There is a
constant $C_3$ independent of $\rho$ and $u$ such that
\begin{equation}
 \left| u(x_0) -\frac{1}{\pi\rho^2} \sum_{v\in V_{int} (x_0,\rho)} F^*(v)u(v)
\right| \leq \frac{C_3 u(x_0)}{\rho},
\end{equation}
where
$F^*(v)=\frac{1}{4} \sum_{[z,v]\in E(G)} c([z,v])|z-v|^2$
is the area of the face of the dual graph $G^*$ corresponding to the vertex
$v\in V_{int}(G)$.
\end{theorem}
\begin{proof}
 Consider the function
\[ p(z)={\cal G}(x_0,z)+\frac{1}{2\pi}(\log(2\rho) +\gamma_{\text{Euler}})
+\frac{|z-x_0|^2-\rho^2}{4\pi \rho^2}.\]
Let $z\in V_{int}(G)$ be an interior vertex.
Consider the chain of faces $f_1,\dots,f_m$ of ${\mathscr D}$ 
which are incident to $z$ in counterclockwise order.
The enumeration
of the faces $f_j$ and of the black
vertices $v_1,\dots,v_m$ and the white vertices $z_1,\dots,z_m$ incident to
these faces can be chosen such that
$f_j$ is incident to $v_{j-1}$, $v_j$, and $z_j$ for $j=1,\dots, m$, where
$v_0=v_m$.
Furthermore, using this enumeration we have
\begin{align*}
&\frac{z_j-z}{|z_j-z|}i=\frac{v_j-v_{j-1}}{|v_j-v_{j-1}|}\quad
\iff\quad
\frac{|v_j-v_{j-1}|}{|z_j-z|} (z_j-z) =-i(v_j-v_{j-1});
\end{align*}
see Figure~\ref{faceFig} with $z_-=z$, $z_+=z_j$, $v_-=v_{j-1}$, $v_+=v_j$.
As
$|z_j-x_0|^2-|z-x_0|^2=-2\Re((\overline{z-x_0})(z_j-z)) + |z_j-z|^2$
we deduce by very similar calculations as in the proof of Lemma~\ref{lem1} that
\begin{align*}
\Delta p (z) &=\Delta {\cal G}(x_0,v)+ \frac{1}{4\pi \rho^2} \sum_{j=1}^m
\underbrace{c([z,z_j])}_{=2f_{\alpha([z,z_j])}'(0)} (|z_j-x_0|^2-|z-x_0|^2) \\
&= -\delta_{x_0}(z) + F^*(z)/(\pi \rho^2).
\end{align*}
Let $v$ be incident to a vertex of $V_\partial (x_0,\rho)$.
Theorem~\ref{theoGreen} implies that
\begin{align*}
 p(v) &= -\frac{1}{2\pi}\log\frac{|v-x_0|}{\rho} +\frac{|v-x_0|^2-\rho^2}{4\pi
\rho^2} + \Od\left(\frac{1}{|v-x_0|^2}\right)
= \Od(1/\rho^2).
\end{align*}
Thus there is a constant $B_1$, independent of $\rho$ and $v$, such that
$p_1(v):= p(v)+B_1/\rho^2\geq 0$ and $|p_1(v)|\leq 2B_1/\rho^2$ for all
vertices $v\in V_{int} (x_0,\rho)$ incident to a vertex of $V_\partial
(x_0,\rho)$.
Applying Green's Identity~\ref{lemGreenId} to $p_1$ and the non-negative
harmonic function $u$, we obtain
\begin{align*}
 u(x_0) -\frac{1}{\pi\rho^2} \sum_{v\in V_{int} (x_0,\rho)} F^*(v)u(v)
&= \sum_{x\in V_{int} (x_0,\rho)} (p_1(x)\Delta u(x) - u(x)\Delta p_1(x)) \\
&= \sum_{[z,q]\in E_\rho} c([z,q])(p_1(z)u(q) -\underbrace{u(z)p_1(q)}_{\geq
0}),\\
&\leq \frac{2B_1}{\rho^2} 4\pi \sum_{q\in V_\partial (x_0,\rho)} u(q) 
\leq \frac{8\pi B_1B_2}{\rho}  u(x_0),
\end{align*}
Here $E_\rho=\{[x,y]\in E(G):  x\in
V_{int}(x_0,\rho),\ y\in V_\partial (x_0,\rho)\}$ and we have used the
estimation
\[\sum_{[x,y]\in E(G)} c([x,y]) \leq \sum_{[x,y]\in E(G)}
c([x,y])|x-y|^2 =4 F^*(x) <4\pi \]
for all fixed vertices $x\in V_{int}(G)$. Furthermore
$ \sum_{y\in V_\partial (x_0,\rho)} u(y) \leq B_2 \rho u(x_0)$
for some constant $B_2>0$ as a consequence of Corollary~\ref{corEstharm}.

For the reverse inequality, note that there is also a constant $B_3$
independent of $\rho$ and $v$ such that
$p_1(v):= p(v)-B_3/\rho^2\leq 0$ and $|p_1(v)|\leq 2B_3/\rho^2$ for all
vertices $v\in V_{int} (x_0,\rho)$ incident to a vertex in $V_\partial
(x_0,\rho)$. 
Combining both estimation proves the claim.
\end{proof}

Theorem~\ref{theoGaussharm} can be interpreted as an analog to the Theorem of
Gauss in potential theory.
Furthermore, we can deduce a discrete version of
H\"older's Inequality for non-negative harmonic function.

\begin{theorem}[H\"older's Inequality]\label{theoHoelder}
 Let  $u:V(x_0,\rho)\to \R$ be a non-negative harmonic function. There is a
constant $C_4$, independent of $\rho$ and $u$, such that
\begin{equation}\label{eqHoelder}
 |u(x_0)-u(x_1)|\leq C_4u(x_0)/\rho
\end{equation}
for all vertices $x_1\in V(x_0,\rho)$ incident to $x_0$.
\end{theorem}

As a corollary of H\"older's Inequality and of Proposition~\ref{propGrhoRand}
we obtain the following result on the
regularity of discrete solutions to elliptic equations.

\begin{lemma}[Regularity Lemma]\label{lemRegIso}
 Let $W\subset V(G)$ and let $u:W\to \R$ be any function.
Set 
$M(u)=\max_{v\in W_{int}}|\Delta u(v)/(4F^*(v))|$, where $F^*(v)$ is the area
of the face dual to $v$ as in Theorem~\ref{theoGaussharm}. Define
$\|\eta\|_W:=\max\{|\eta(z)|:z\in W\}$.
There are constants $C_5,C_6>0$, independent of $W$ and $u$, such that
\begin{equation}\label{eqReg}
 |u(x_0)-u(x_1)|\rho \leq C_5\|u\|_W +\rho^2 C_6 M(u)
\end{equation}
for all vertices $x_1\in W$ incident to $x_0\in W_{int}$, where $\rho$ is the
Euclidean
distance of $x_0$ to the boundary $W_\partial$.
\end{lemma}
\begin{proof}
Let $x_1\in W$ be a fixed vertex incident to $x_0$.

First we suppose that $\rho\geq 4$.
Consider the auxiliary function $f(z)=M(u)|z-x_0|^2$. Since $|x_1-x_0|<2$,
we obviously have
\[ |f(x_0)-f(x_1)|=M(u)|x_1-x_0|^2 \leq 4M(u).\]

Let $h:V(x_0,\rho)\to \R$ be the unique harmonic function with boundary values
$h(v)=u(v)+f(v)$ for $v\in V_\partial (x_0,\rho)$. H\"older's
Inequality~\eqref{eqHoelder} and the Maximum Principle~\ref{MaxPrinzip} imply
that
\[|h(x_0)-h(x_1)|\rho \leq B_1\|h\|_{V(x_0,\rho)} \leq B_1(\|u\|_W
+M(u)\rho^2) \]
for some constant $B_1$ independent of $h$, $\rho$, $x_0$, $x_1$.

Next consider $s=u+f-h$ on $V(x_0,\rho)$. Then 
\[\begin{cases} \Delta s =\Delta u +4 F^* M(u)\geq 0   &\text{on }
V_{int}(x_0,\rho), \\
 s(v)=0 &\text{for } v\in V_\partial (x_0,\rho).
\end{cases}\] 
The Maximum Principle~\ref{MaxPrinzip} implies $s\leq 0$.
Green's Identity~\ref{lemGreenId} gives
\begin{multline*}
 s(x_0) +\sum_{v\in V_{int} (x_0,\rho)} {\cal G}_{x_0,\rho}(x_0,v)\Delta s(v) \\
= \sum_{v\in V_{int} (x_0,\rho)} ({\cal G}_{x_0,\rho}(x_0,v)\Delta s(v) - s(v)
\Delta {\cal G}_{x_0,\rho}(x_0,v) ) \\
= \sum_{[p,q]\in E_\rho} c([p,q])({\cal G}_{x_0,\rho}(x_0,p)s(q) -s(p){\cal
G}_{x_0,\rho}(x_0,q)) 
=0,
\end{multline*}
where $E_\rho= \{[p,q]\in E(G) : p\in V_{int}(x_0,\rho), q\in
V_\partial (x_0,\rho)\}$.
Analogously,
\[s(x_1) +\sum_{v\in V_{int} (x_0,\rho)} {\cal G}_{x_0,\rho}(x_1,v)\Delta s(v)
=0.\]
Using the estimation $\Delta s(v) \leq 8F^*(v)M(u)$ we deduce that
\begin{align*}
 |s(x_0)-s(x_1)| &\leq \sum_{v\in V_{int} (x_0,\rho)} |{\cal
G}_{x_0,\rho}(x_0,v)
-{\cal G}_{x_0,\rho}(x_1,v)|8F^*(v)M(u).
\end{align*}
Now Proposition~\ref{propGrhoRand} implies that
\begin{align*}
|s(x_0)-s(x_1)| &\leq 8B_2 M(u) \sum_{v\in V_{int} (x_0,\rho)}
\frac{F^*(v)}{|v-x_0|+1}
\leq 8B_2 M(u)
B_3 \rho,
\end{align*}
where $B_2$ and $B_3$ are constants independent of $s$, $\rho$, $x_0$, $x_1$.

Combining the above estimations for $f$, $h$, and $s$, we finally obtain
\begin{align*}
|u(x_0)-u(x_1)|\rho &\leq |s(x_0)-s(x_1)- (f(x_0)-f(x_1)) + h(x_0)-h(x_1)|\rho
\\
 &\leq B_1\|u\|_W + \rho^2 (4+B_1+8B_2B_3) M(u).
\end{align*}
This implies the claim for $\rho\geq 4$.
For $\rho<4$ inequality~\eqref{eqReg} can be deduced from
\[ -4F^*(x_0) M(u) \leq \Delta u (x_0)= \sum_{[x_0,v]\in E(G)} c([x_0,v])
(u(v)-u(x_0)) \leq 4F^*(x_0) M(u).\]
using $F^*(x_0)\leq \pi$, $\sum_{[x_0,v]\in E(G)} c([x_0,v]) \leq 4\pi$, and
the uniform boundedness of the weights $c(e)$.
\end{proof}

\small
\bibliography{../dissertation,../paperbibliography}

\providecommand{\bysame}{\leavevmode\hbox to3em{\hrulefill}\thinspace}
\providecommand{\MR}{\relax\ifhmode\unskip\space\fi MR }
\providecommand{\MRhref}[2]{%
  \href{http://www.ams.org/mathscinet-getitem?mr=#1}{#2}
}
\providecommand{\href}[2]{#2}
\begin{thebibliography}{BSSZ08}

\bibitem[AB00]{AB00}
S.~I. Agafonov and A.~I. Bobenko, \emph{Discrete {$Z^\gamma$} and {P}ainlev\'e
  equations}, Internat. Math. Res. Notices \textbf{4} (2000), 165--193.

\bibitem[BH03]{BH03}
A.~I. Bobenko and T.~Hoffmann, \emph{Hexagonal circle patterns and integrable
  systems: {P}atterns with constant angles}, Duke Math. J. \textbf{116} (2003),
  525--566.

\bibitem[BMS05]{BMS05}
A.~I. Bobenko, Ch. Mercat, and Yu.~B. Suris, \emph{Linear and nonlinear
  theories of discrete analytic functions. {I}ntegrable structure and
  isomonodromic {G}reen's function}, J. reine angew. Math. \textbf{583} (2005),
  117--161.

\bibitem[BS04]{BS02}
A.~I. Bobenko and B.~A. Springborn, \emph{Variational principles for circle
  patterns and {K}oebe's theorem}, Trans. Amer. Math. Soc. \textbf{356} (2004),
  659--689.

\bibitem[BS08]{BS08}
A.~I. Bobenko and Yu.~B. Suris, \emph{Discrete differential geometry. {T}he
  integrable structure}, to appear in 2008.

\bibitem[BSSZ08]{BSSZ08}
A.~I. Bobenko, P.~Schr\"oder, J.~M. Sullivan, and G.~M. Ziegler (eds.),
  \emph{Discrete differential geometry}, Oberwolfach Seminars, vol.~38,
  Birkh\"auser, Basel, 2008.

\bibitem[B{\"u}c07]{diss}
U.~B{\"u}cking, \emph{Approximation of conformal mappings by circle patterns
  and discrete minimal surfaces}, Ph.D. thesis, Technische Universit\"at
  Berlin, 2007, published online at
  \url{http://opus.kobv.de/tuberlin/volltexte/2008/1764/}.

\bibitem[CR92]{CR92}
I.~Carter and B.~Rodin, \emph{An inverse problem for circle packing and
  conformal mapping}, Trans. Amer. Math. Soc. \textbf{334} (1992), 861--875.

\bibitem[DK85]{DK}
M.~Duneau and A.~Katz, \emph{Quasiperiodic patterns}, Phys. Rev. Lett.
  \textbf{54} (1985), 2688--2691.

\bibitem[Duf53]{Du53}
R.~J. Duffin, \emph{Discrete potential theory}, Duke Math. J. \textbf{20}
  (1953), 233--251.

\bibitem[Duf68]{Du68}
\bysame, \emph{Potential theory on a rhombic lattice}, J. Combin. Th.
  \textbf{5} (1968), 258--272.

\bibitem[GR86]{GR}
F.~G{\"a}hler and J.~Rhyner, \emph{Equivalence of the generalized grid and
  projection methods for the construction of quasiperiodic tilings}, J. Phys. A
  \textbf{19} (1986), 267--277.

\bibitem[He99]{He99}
Z.-X. He, \emph{Rigidity of infinite disk patterns}, Ann. of Math. \textbf{149}
  (1999), 1--33.

\bibitem[HS96]{HeSch96}
Z.-X. He and O.~Schramm, \emph{On the convergence of circle packings to the
  {R}iemann map}, Invent. Math. \textbf{125} (1996), 285--305.

\bibitem[HS98]{HeSch98}
\bysame, \emph{The {$C^\infty$}-convergence of hexagonal disk packings to the
  {R}iemann map}, Acta Math. \textbf{180} (1998), 219--245.

\bibitem[Ken02]{Ke02}
R.~Kenyon, \emph{The {L}aplacian and {D}irac operators on critical planar
  graphs}, Invent. math. \textbf{150} (2002), 409--439.

\bibitem[LD07]{LD07}
S.-Y. {Lan} and D.-Q. {Dai}, \emph{The {$C^\infty$-convergence} of {SG} circle
  patterns to the {R}iemann mapping}, J. of Math. Analysis and Appl.
  \textbf{332} (2007), 1351--1364.

\bibitem[Mat05]{Ma05}
D.~Matthes, \emph{Convergence in discrete {C}auchy problems and applications to
  circle patterns}, Conform. Geom. Dyn. \textbf{9} (2005), 1--23.

\bibitem[Mer01]{Me}
Ch. Mercat, \emph{Discrete {R}iemann surfaces and the {I}sing model}, Commun.
  Math. Phys. \textbf{218} (2001), 177--216.

\bibitem[Riv94]{Ri94}
I.~Rivin, \emph{Euclidean structures on simplicial surfaces and hyperbolic
  volume}, Ann. of Math. \textbf{139} (1994), 553--580.

\bibitem[RS87]{RS87}
B.~Rodin and D.~Sullivan, \emph{The convergence of circle packings to the
  {R}iemann mapping}, J. Diff. Geom. \textbf{26} (1987), 349--360.

\bibitem[SC97]{SC97}
L.~Saloff-Coste, \emph{Some inequalities for superharmonic functions on
  graphs}, Potential Anal. \textbf{6} (1997), 163--181.

\bibitem[Sch97]{Sch97}
O.~Schramm, \emph{Circle patterns with the combinatorics of the square grid},
  Duke Math. J. \textbf{86} (1997), 347--389.

\bibitem[Sen95]{Se}
M.~Senechal, \emph{Quasicrystals and geometry}, Cambridge Univ. Press, 1995.

\bibitem[Spr03]{Spr03}
B.~A. Springborn, \emph{Variational principles for circle patterns}, Ph.D.
  thesis, Technische Universit\"at Berlin, 2003, published online at
  \url{http://opus.kobv.de/tuberlin/volltexte/2003/668/}.

\bibitem[Ste05]{St05}
K.~Stephenson, \emph{Introduction to circle packing: the theory of discrete
  analytic functions}, Cambridge University Press, New York, 2005.

\bibitem[Thu85]{Thu85}
B.~Thurston, \emph{The finite {R}iemann mapping theorem}, Invited address at
  the International Symposioum in Celebration of the proof of the Bieberbach
  Conjecture, Purdue University, March 1985.

\end{thebibliography}

\end{document}